\documentclass[12pt]{article}
\usepackage{amsthm}
\usepackage{amsfonts}
\usepackage{epsfig}

\newtheorem{theorem}{Theorem}
\newtheorem{problem}{Problem}

\newtheorem{lemma}[theorem]{Lemma}

\newcounter{claims}
\newenvironment{claims}{\refstepcounter{claims}\par\medskip\noindent%
{{\bf (\theclaims)}~~}}{\par\medskip}
\newcommand{\claim}[2]{\begin{claims}{\em #2}\label{#1}\end{claims}}
\newcommand{\refclaim}[1]{(\ref{#1})}
\newcommand{\cin}{\mathop{\mathrm{Int}}}
\newcommand{\cex}{\mathop{\mathrm{Ext}}}
\newcommand{\dom}{\mbox{dom}}
\newcommand{\ob}[1]{{\rm O}$_{\rm #1}$}  
\newcommand{\gf}{G_{\varphi}}
\newcommand{\lf}{L_{\varphi}}
\newcommand{\vf}{\varphi}

\title{$5$-choosability of graphs with crossings far apart\thanks{This research was supported by the CZ-SL bilateral project MEB 091037.}}

\author{Zden\v{e}k Dvo\v{r}\'ak\thanks{Computer Science Institute of Charles University, Prague, Czech Republic.
E-mail: {\tt rakdver@iuuk.mff.cuni.cz}.
Supported by Institute for Theoretical Computer Science (ITI), project 1M0545 of Ministry of Education of Czech Republic,
and by project GA201/09/0197 (Graph colorings and flows: structure and applications) of Czech Science Foundation.
Revised under the support of project 14-19503S (Graph coloring and structure) of Czech Science Foundation.}
\and
Bernard Lidick\'y\thanks{
Iowa State University, Ames IA, USA.
E-mail: {\tt lidicky@iastate.edu}.
Supported in part by NSF grants DMS-1266016 and DMS-1600390.
  }
\and
Bojan Mohar\thanks{Department of Mathematics, Simon Fraser University, Burnaby, B.C. V5A 1S6.
  E-mail: {\tt mohar@sfu.ca}.
  Supported in part by an NSERC Discovery Grant (Canada),
  by the Canada Research Chair program, and by the
  Research Grant P1--0297 of ARRS (Slovenia).}~\thanks{On leave from:
  IMFM \& FMF, Department of Mathematics, University of Ljubljana, Ljubljana,
  Slovenia.}
}
\date{}

\begin{document}
\maketitle

\begin{abstract}
We give a new proof of the fact that every planar graph is $5$-choosable,
and use it to show that every graph drawn in the plane so that the distance
between every pair of crossings is at least $15$ is $5$-choosable. At the same time we may allow some vertices to have lists of size four only, as long as they are far apart and far from the crossings.
\end{abstract}

Thomassen~\cite{thomassen1994} showed that every planar graph is $5$-choosable.  
We prove a strengthening of this result, allowing some crossings.  Suppose that a graph $G$ is drawn in the plane
with some crossings, and for $i\in\{1,2\}$, let $u_iv_i$ and $x_iy_i$ be two edges of $G$ crossing each other.
The \emph{distance} between the crossing of $u_1v_1$ with $x_1y_1$ and the crossing of $u_2v_2$ with $x_2y_2$
is the length of the shortest path in $G$ with one end in the set $\{u_1,v_1,x_1,y_1\}$ and the other end in the set $\{u_2,v_2,x_2,y_2\}$.
Our main result is the following.

\begin{theorem}\label{thm-main}
Every graph drawn in the plane so that the distance between every pair of crossings is at least\/ $15$ is $5$-choosable.
\end{theorem}

Let us recall that a {\em list assignment\/} $L$ for $G$ is a function that
assigns to each vertex of $G$ a set $L(v)$, called the {\em list of admissible
colors} for $v$. An {\em $L$-coloring} is a choice of a color $\vf(v)\in L(v)$
for each $v\in V(G)$ such that no two adjacent vertices receive the same color.
The graph is {\em $k$-choosable} if it admits an $L$-coloring for every list
assignment $L$ with $|L(v)|\ge k$ for every $v\in V(G)$.

The main idea of Thomassen's beautiful proof of $5$-choosability of planar graphs is to establish
the following more general statement.

\begin{theorem}[Thomassen~\cite{thomassen1994}]\label{thm-thom}
Let $G$ be a plane graph with the outer face $F$, $xy$ an edge of $F$, and $L$ a list assignment such that $|L(v)|\ge 5$ for $v\in V(G)\setminus V(F)$,
$|L(v)|\ge 3$ for $v\in V(F)\setminus\{x,y\}$, $|L(x)|=|L(y)|=1$ and $L(x)\neq L(y)$.  Then $G$ is $L$-colorable.
\end{theorem}

Let us note that the lists of $x$ and $y$ of size $1$ give a precoloring of a path of length $1$ in the outer face of $G$.
Naturally, one might try to prove Theorem~\ref{thm-main} by showing a variant of Theorem~\ref{thm-thom} allowing distant crossings.
Unfortunately, almost any attempt to alter the statement of Theorem~\ref{thm-thom} (e.g., by allowing more than two vertices to be precolored,
allowing lists of size $2$ subject to some constraints, allowing some crossings in the drawing, etc.) fails with infinitely many counterexamples.
To overcome this obstacle, we give a new proof of $5$-choosability of planar graphs, see Theorem~\ref{thm-basic} in Section~\ref{sec-planar},
which turns out to be more robust with respect to some strengthenings of the planar 5-choosability theorem.
The proof of Theorem~\ref{thm-basic} is inspired by Thomassen's proof~\cite{thomassen1995-34} of $3$-choosability of planar graphs of girth $5$.

In the course of the proof of Theorem~\ref{thm-main}, it is convenient to allow in addition to crossings also some other irregularities, such as vertices
with fewer than $5$ available colors.  Hence, we actually obtain the following stronger statement (the distance between a vertex $z$ and a crossing of edges $uv$ and $xy$ is the
length of the shortest path from $z$ to $\{u,v,x,y\}$).

\begin{theorem}\label{thm-main0}
Let $G$ be a graph drawn in the plane with some crossings and let $N\subseteq V(G)$ be a set of vertices such that the distance between any pair of crossed edges is at least\/ $15$, the distance between any crossing and a vertex in $N$ is at least\/ $13$, and the distance between any two vertices in $N$ is at least\/ $11$.
Then $G$ is $L$-colorable for any list assignment $L$ such that $|L(v)|=4$ for $v\in N$ and $|L(v)|\ge 5$ for $v\in V(G)\setminus N$.
\end{theorem}

Some distance condition on the crossings in Theorem~\ref{thm-main} is necessary, even if we allow
only three crossings, as shown by $K_6$. On the other hand, it was proved in
\cite{LidDvoSkre} and independently also in \cite{CaHa11} that the distance
requirement is not needed, if we have at most two crossings. The inductive
proof of Theorem~\ref{thm-main0} involves a stronger inductive hypothesis that
is stated later as Theorem \ref{thm-maingen} and in particular also implies the
above-mentioned result from~\cite{LidDvoSkre, CaHa11}, see Section~\ref{sec-main}.

\begin{theorem}[\cite{LidDvoSkre, CaHa11}]
\label{thm-2x}
Every graph whose crossing number is at most two is $5$-choosable.
\end{theorem}

The proof of Theorem \ref{thm-2x} is given at the end of Section~\ref{sec-main}.
Another special case of Theorem \ref{thm-main0} is the following.

\begin{theorem}\label{thm-mainalt}
Let $G$ be a planar graph and $N\subseteq V(G)$ a set of vertices such that the distance between any pair of vertices in $N$ is at least\/ $11$.
Then $G$ is $L$-colorable for any list assignment $L$ such that $|L(v)|=4$ for $v\in N$ and $|L(v)|\ge 5$ for $v\in V(G)\setminus N$.
\end{theorem}

The last result is motivated by the result of Voigt~\cite{voigt1993} showing that not all planar graphs are $4$-choosable.
Furthermore, it is related to the following problem of Albertson~\cite{Alb98}:

\begin{problem}\label{quest-albertson}
Does there exist a constant $d$ such that whenever $G$ is a planar graph with list assignment $L$ that
gives a list of size one or five to each vertex and the distance between any pair of vertices with list
of size one is at least $d$, then $G$ is $L$-colorable?
\end{problem}

Building upon the ideas of this paper, we were able to give a positive answer to this problem, which we present in a separate paper~\cite{LidDvoMohPos}.

We start with giving the new proof of $5$-choosability of planar graphs in Section~\ref{sec-planar}.  After introducing some
terminology in Section~\ref{sec-term}, we formulate a strengthening of Theorem~\ref{thm-main0} suitable for the inductive argument
in Section~\ref{sec-main}.  We give a proof of this more
general statement in Section~\ref{sec-maingen}.

\section{Planar graphs}\label{sec-planar}

Let $P$ be a path or a cycle.
The {\em length} $\ell(P)$ of $P$ is the number of its edges, i.e., a path of length $l$ has $l+1$ vertices and a cycle of length $l$ has $l$ vertices.  Given a graph $G$ and a cycle $K\subseteq G$, an edge $uv$ of $G$ is a {\em chord\/} of $K$ if $u,v\in V(K)$, but $uv$ is not an edge of $K$.  For an integer $k\ge 2$, a path $v_0v_1\ldots v_k$ is
a {\em $k$-chord\/} if $v_0,v_k\in K$ and $v_1, \ldots, v_{k-1}\not\in V(K)$.  We define a chord to be a {\em $1$-chord}.  If $G$ is a plane graph, let
$\cin_K(G)$ be the subgraph of $G$ consisting of the vertices and edges drawn inside the closed disc bounded by $K$,
and $\cex_K(G)$ the subgraph of $G$ obtained by removing all vertices and edges drawn inside the open disc bounded by $K$.
In particular, $K = \cin_K(G) \cap \cex_K(G)$. Note that
each $k$-chord of $K$ belongs to exactly one of $\cin_K(G)$ and $\cex_K(G)$.  We say that a cycle $K$ is \emph{separating}
if $V(\cin_K(G))\neq V(K)\neq V(\cex_K(G))$.
If the cycle $K$ bounds the outer face of $G$ and $Q$ is a $k$-chord of $K$, let $C_1$ and $C_2$ be the two cycles in $K\cup Q$ that contain $Q$. Then the subgraphs $G_1=\cin_{C_1}(G)$ and $G_2=\cin_{C_2}(G)$ are the {\em $Q$-components} of~$G$.

As we mentioned earlier, Thomassen's Theorem \ref{thm-thom} does not extend to the case when we have a precolored path of length two. However, if we strengthen the condition on the list sizes of the other vertices on the outer face, such an extension is possible.

\begin{theorem}\label{thm-basic}
Let $G$ be a plane graph with the outer face $F$, $P$ a subpath of $F$ of length at most two and $L$ a list assignment such that the following conditions are satisfied:
\begin{itemize}
\setlength{\itemsep}{0pt}
\item[\rm (i)] $|L(v)|\ge 5$ for $v\in V(G)\setminus V(F)$,
\item[\rm (ii)] $|L(v)|\ge 3$ for $v\in V(F)\setminus V(P)$,
\item[\rm (iii)] $|L(v)|=1$ for $v\in V(P)$,
\item[\rm (iv)] no two vertices with lists of size three are adjacent in $G$,
\item[\rm (v)] $L$ gives a proper coloring to the subgraph induced by $V(P)$, and
\item[\rm (vi)] if $P=uvw$ has length two and $x$ is a common neighbor of $u$, $v$ and $w$, then $L(x)\neq L(u)\cup L(v)\cup L(w)$.
\end{itemize}
Then $G$ is $L$-colorable.
\end{theorem}

\begin{proof}
Suppose for a contradiction that the claim is false, and let $G$ be a counterexample with $|V(G)|+|E(G)|$ the smallest possible,
and subject to that, with the longest path $P$ and with the minimum size of the lists (while satisfying (i)--(vi)).
It is clear that $G$ is connected and that every vertex $v\in V(G)$ satisfies $\deg(v)\ge |L(v)|$.

Furthermore, $G$ is $2$-connected: otherwise, let $v$ be a cut-vertex and let $G_1$
and $G_2$ be subgraphs of $G$ such that $G_1\cup G_2=G$, $V(G_1)\cap V(G_2)=\{v\}$ and $|V(G_1)|, |V(G_2)|>1$.  If $v\in V(P)$, then by the minimality of $G$ there exist $L$-colorings of $G_1$ and $G_2$, and
these colorings together give an $L$-coloring of $G$.  Otherwise, we may assume by symmetry that $P\subseteq G_1$. Consider an $L$-coloring $\vf$ of $G_1$.
Let $L_2$ be the list assignment for $G_2$ such that $L_2(u)=L(u)$ for $u\neq v$ and $L_2(v)=\{\vf(v)\}$.
By the minimality of $G$, $G_2$ is $L_2$-colorable, and this coloring together with $\vf$ gives an $L$-coloring of $G$.

Every triangle $T$ in $G$ bounds a face: otherwise, first color the subgraph $\cex_T(G)$
and then extend the coloring to $\cin_T(G)$.  A similar argument shows that $G$ contains no separating $4$-cycles;
otherwise, consider such a $4$-cycle $K=k_1k_2k_3k_4$, and let $\vf$ be an $L$-coloring of $\cex_K(G)$.
Let $G'=\cin_K(G)$.  Since $K$ is separating, we have $V(G')\neq V(K)$, and since every triangle bounds a face, we conclude that $K$ has
no chord in $G'$.  Let $L'$ be the list assignment for $G'-k_1$ such that $L'(z)=\{\vf(z)\}$ for $z\in \{k_2,k_3,k_4\}$,
$L'(z)=L(z)\setminus \{\vf(k_1)\}$ if $z\not\in\{k_2,k_4\}$ is a neighbor of $k_1$ and $L'(z)=L(z)$ if $z$ is any other vertex.
By the minimality of $G$, the graph $G'-k_1$ is $L'$-colorable, and this coloring together with $\vf$ gives an $L$-coloring of $G$.

Since $G$ is 2-connected, its outer face is bounded by a cycle, which we denote by $F$ as well.
Next, we show that $F$ has no chords. Otherwise, let $uv$ be a chord of $F$ and let $G_1$ and $G_2$ be the $uv$-components of $G$.
If $P\subseteq G_1$, then we first color $G_1$ and then extend the coloring to $G_2$.  The case that $P\subseteq G_2$ is symmetric.  It follows that $P$ has length two and all the chords of $F$ are incident
with its middle vertex.  Let $P=z_1uz_2$, where $z_i\in V(G_i)$ for $i\in\{1,2\}$.  Let $\vf$ be an $L$-coloring of $G_1$ and let $L_2$ be the list assignment for $G_2$ such that
$L_2(z)=L(z)$ for $z\neq v$ and $L_2(v)=\{\vf(v)\}$. Since $G$ is not $L$-colorable, $G_2$ is not $L_2$-colorable.
By the minimality of $G$, either $v$ is adjacent to $z_2$, or $u$, $v$ and $z_2$ have a common neighbor $w$ with list of size three (which means, in particular, that $w\in V(F)$). 
Since every chord of $G$ is incident with $u$, the edge $vz_2$ or $vw$ belongs to $F$.  Since every triangle bounds a face, we conclude that $v$ has
degree two in $G_2$. By symmetry, $v$ has degree two in $G_1$ as well, and thus $v$ has degree three in $G$.  It follows that $|L(v)|=3$,
and thus $v$ cannot be adjacent to any other vertex with list of size three.  In particular, we cannot have the case with the vertex $w$.  We conclude that $v$ is adjacent to $z_1$ and $z_2$ and
$V(G)=\{u,v,z_1,z_2\}$.  However, $L(v)\neq L(u)\cup L(z_1)\cup L(z_2)$ by (vi), and thus $G$ is $L$-colorable.  This contradiction proves that $F$ has no chords.

Similarly, we have the following property:

\claim{cl-bas-2chord}{Let $uvw$ be a $2$-chord of $F$ and let $G_1$ and $G_2$ be $uvw$-components of $G$.  If $P\subseteq G_1$, then
either $u$ and $w$ are adjacent and $G_2$ is equal to the triangle $uvw$, or there exists a vertex $x$ such that $V(G_2)=\{u,v,w,x\}$,
$|L(x)|=3$ and $x$ is adjacent to $u$, $v$ and $w$.}

If $\ell(P)<2$, then it is easy to see that we can precolor $2-\ell(P)$ more vertices of $F$ without violating (vi). Thus, we may assume that $\ell(P)=2$.  Let $P=p_0p_1p_2$.
Suppose that $p_0$, $p_1$ and $p_2$ have a common neighbor $v$.  If $v\in V(F)$, then $V(G)=\{p_0,p_1,p_2,v\}$ and $G$ is $L$-colorable.
If $v\not\in V(F)$, then $v$ has degree at most four in $G$ by \refclaim{cl-bas-2chord} and thus $\deg(v)<|L(v)|$, which is a contradiction.
Therefore, $p_0$, $p_1$ and $p_2$ have no common neighbor.

Furthermore, $\ell(F)\ge 6$:  If $\ell(F)=3$, then
we remove one vertex of $F$ and remove its color from the lists of all its neighbors, and observe that the resulting graph is a smaller counterexample
to Theorem~\ref{thm-basic}.  In the case when $\ell(F)=4$, then similarly color and remove the vertex of $V(F)\setminus V(P)$.  Finally, suppose that
$\ell(F)=5$.  Let $\vf$ be an arbitrary $L$-coloring of $F=p_2p_1p_0v_1v_2$.  Remove $v_1$ and $v_2$ from $G$ and remove their colors according
to $\vf$ from the lists of their neighbors, obtaining a graph $G'$ with the list assignment $L'$.  Since every triangle in $G$ bounds a face,
at most one vertex in $G'$ has list of size three.  Since $p_0$, $p_1$ and $p_2$ have no common neighbor and $p_0$ is not adjacent to $p_2$,
$G'$ with the list assignment $L'$ is a smaller counterexample to Theorem~\ref{thm-basic}, which is a contradiction.

Let $F=p_2p_1p_0v_1v_2v_3v_4\ldots$.  If $\ell(F)=6$, then we set $v_4=p_2$.
We may assume that $|L(v_1)|=3$ or $|L(v_2)|=3$, since otherwise we can remove a color from the list of $v_1$.
Let us consider a set $X\subseteq V(F)\setminus V(P)$ and a partial $L$-coloring $\vf$ of $X$ that are defined as follows:
\begin{itemize}
\item[(X1)] If $|L(v_1)|=3$ and $|L(v_3)|\neq 3$, then $X=\{v_1\}$ and $\vf(v_1)\in L(v_1)\setminus L(p_0)$ is chosen arbitrarily.
\item[(X2)] If $|L(v_1)|=3$ and $|L(v_3)|=3$, then $X=\{v_1,v_2\}$ and $\vf$ is chosen so that $\vf(v_2)\in L(v_2)\setminus L(v_3)$
and $\vf(v_1)\in L(v_1)\setminus(L(p_0)\cup \{\vf(v_2)\})$.
\item[(X3)] If $|L(v_2)|=3$, and either $|L(v_4)|\neq 3$ or $|L(v_3)|\geq 5$, then $X=\{v_2\}$ and $\vf(v_2)\in L(v_2)$ is chosen arbitrarily.
\item[(X4)] If $|L(v_2)|=3$, $|L(v_3)|=4$ and $|L(v_4)|=3$, then:
\begin{itemize}
\item[(X4a)] If $v_1$, $v_2$ and $v_3$ do not have a common neighbor or $|L(v_1)|\ge5$, then $X=\{v_2, v_3\}$ and $\vf$ is chosen
so that $\vf(v_3)\in L(v_3)\setminus L(v_4)$ and $\vf(v_2)\in L(v_2)\setminus \{\vf(v_3)\}$.
\item[(X4b)] If $v_1$, $v_2$ and $v_3$ have a common neighbor and $|L(v_1)|=4$, then $X=\{v_1, v_2, v_3\}$ and $\vf$ is chosen
so that $\vf(v_3)\in L(v_3)\setminus L(v_4)$, $\vf(v_1)\in L(v_1)\setminus L(p_0)$ and either at least one of $\vf(v_1)$ and $\vf(v_3)$
does not belong to $L(v_2)$, or $\vf(v_1)=\vf(v_3)$.  The vertex $v_2$ is left uncolored.
\end{itemize}
\end{itemize}
For later reference, Figure \ref{fig-definition_of_X} shows the subcases used in the definition of $X$ and $\vf$.

\begin{figure}
\begin{center}
\includegraphics[width=100mm]{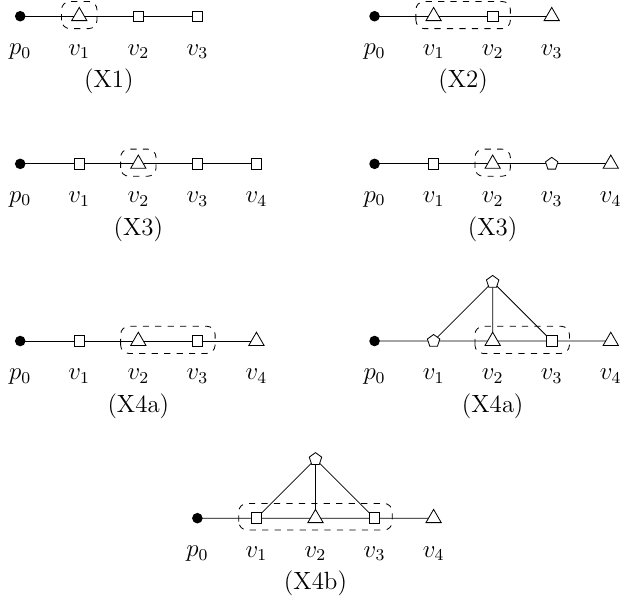}
\end{center}
\caption{Subcases in the definition of $X$. Triangle vertices represent lists of size 3, square vertices list of size $\ge4$. Encircled vertices are in $X$.}
\label{fig-definition_of_X}
\end{figure}

Let $G'=G-X$ and let $L'$ be the list assignment obtained from $L$ by removing the colors of the vertices of $X$ according to $\vf$ from the lists of their
neighbors (if a vertex of $X$ is not colored according to $\vf$, we do not remove any colors for it).  Observe that any $L'$-coloring of $G'$
can be extended to an $L$-coloring of $G$, thus $G'$ is not $L'$-colorable.  By the minimality of $G$, this implies that $G'$ violates the assumptions
of Theorem~\ref{thm-basic}.  Since $F$ has no chords, the choice of $X$ and $\vf$ implies that every vertex of $V(G')\setminus V(P)$ has list of size at least three.
Since $p_0$ is not adjacent to $p_2$, and since $p_0$, $p_1$, and $p_2$ do not have a common neighbor in $G$, it follows that the conditions
(v) and (vi) are satisfied by $G'$ with the list assignment $L'$.  We conclude that (iv) is false, i.e., $G'$ contains adjacent vertices $u$ and $v$
such that $|L'(u)|=|L'(v)|=3$.

Since $F$ has no chords, the choice of $X$ ensures that at most one of $u$ and $v$ belongs to $V(F)$;
hence, we can assume that $v\not\in V(F)$ and $v$ has two neighbors in $\dom(\vf)$.  In particular, $X$ and $\vf$ were chosen according to the cases
(X2) or (X4).  Since $G$ contains no separating cycles of length at most $4$,
we conclude that $u$ has at most one neighbor in $\dom(\vf)$.  Since $|L'(u)|=3$, it follows that $u\in V(F)$.
Let $x\in X$ be the neighbor of $v$ such that the
distance between $u$ and $x$ in $F-P$ is as large as possible.  By \refclaim{cl-bas-2chord} applied to the $2$-chord $xvu$,
we conclude that the $xvu$-component of $G$ that does not contain $P$ consists of $xvu$ and a vertex $z$ adjacent to $x$, $v$ and $u$
with $|L(z)|=3$.  It follows that $|L(u)|>3$, and since $|L'(u)|=3$, we have $z\in X$ and $|L(u)|=4$.  The inspection of the choice
of $X$ shows that (X4) holds, i.e., $u=v_1$, $z=v_2$ and $x=v_3$.  However, note that the condition of (X4b) holds; hence $u\in X$,
contrary to the assumption that $u\in V(G')$.  This completes the proof of Theorem~\ref{thm-basic}.
\end{proof}

\section{Drawings with crossings}\label{sec-term}

In this section, we introduce terminology concerning graphs drawn in the plane with crossings.

Let $G$ be a graph.
A {\em drawing $\cal G$} of a graph $G$ in the plane consists of a set
${\cal V}=\{p_v \mid v\in V(G)\}$ of distinct points in the plane and a set of simple polygonal curves ${\cal E}=\{c_e \mid e\in E(G)\}$ such that
\begin{itemize}
\item if $uv\in E(G)$, then $p_u$ and $p_v$ are the endpoints of $c_{uv}$,
\item no internal point of any $c_e\in {\cal E}$ belongs to ${\cal V}$, and
\item any point in the plane that does not belong to ${\cal V}$ is contained in at most two of the curves in ${\cal E}$, and any two curves in ${\cal E}$ have at most one point in common.
\end{itemize}
A {\em crossing} of ${\cal G}$ is a point in the plane that belongs to two of the curves in ${\cal E}$, but not to ${\cal V}$.
An edge $e$ is {\em incident with the crossing $x$} if $x\in c_e$.
An edge $e$ is {\em crossed} if it is incident with some crossing, and {\em non-crossed} otherwise.
For a crossing $x$, we define $G_x$ to be the graph consisting of the two edges incident with $x$.
Two vertices of $G$ are {\em crossing-adjacent} if they belong to $G_x$ for some crossing $x$ and are not adjacent in $G_x$.
By a slight abuse of terminology, we will generally use $e$ to refer to the curve $p_e$ and $v$ to refer to the point $p_v$
in the drawing, and we will refer by $G$ to both the graph $G$ and its drawing ${\cal G}$.

Removal of $\bigcup {\cal E}$ splits the plane into several connected subsets, which we call {\em faces} of $G$. Slightly abusing the terminology again,
we sometimes identify a face with its boundary and hence speak about the vertices, edges and crossings of the face.
Let us remark that due to the existence of crossings, the boundary of a face may have somewhat more complicated structure than in the plane case.
Let $u$ and $v$ be two consecutive vertices in the boundary of a face $f$.  The part of the boundary of $f$ between $u$ and $v$ is either formed by an edge (which is non-crossed),
or by parts of two crossing edges $ux$ and $vy$ (and $u$ is then crossing-adjacent to $v$).  In the latter case, $x$ and $y$ may or may not be incident with $f$ themselves
(and even if say $x$ is incident with $f$, there may be a part of the edge $ux$ that is not contained in the boundary of $f$).
Let us remark that if both parts of edges $ux$ and $vy$ between the crossing and $x$ and $y$ are contained in the boundary of $f$,
then we can redraw $G$ to eliminate the crossing by ``flipping'' the part of the drawing of $G$ that contains $v$ and $x$.

By a \emph{chord} or \emph{$1$-chord} of the face $f$, we mean an edge $e$ of $G$ such that both endpoints of $e$ are contained in the boundary of $f$,
and either $e$ is crossed, or $e$ is non-crossed and not contained in the boundary of $f$.
A \emph{$k$-chord} of $f$ for $k\ge 2$ is a path $v_0v_1\ldots v_k$ such that $v_0$ and $v_k$ are contained in the boundary of $f$ and
$v_i$ is not contained in the boundary of $f$ for $i=1,\ldots, k-1$,
and the edges of the path do not cross each other.  Note that edges of a $k$-chord may or may not be crossed by other
edges not belonging to the $k$-chord.

For a cycle $K$ whose edges do not cross each other in $G$, let
$\cin_K(G)$ denote the subgraph of $G$ consisting of the vertices and edges drawn fully inside the closed disc bounded by $K$ (i.e., excluding the
edges crossing with the edges of $K$), and let $\cex_K(G)$ denote the subgraph of $G$ obtained by removing all vertices and edges
intersecting the open disc bounded by $K$ (so again, $\cex_K(G)$ does not contain the edges crossing with the edges of $K$).

Suppose that $G$ is $2$-connected, and consider a $k$-chord $Q$ of the outer face $F$ of $G$.
Let $c$ be a closed curve consisting of $Q$ and of an arbitrary curve in the interior of $F$ joining the endpoints of $Q$.
Let $G_1$ be the subgraph of $G$ consisting of the vertices and edges drawn fully inside the closed disc bounded by $c$,
and let $G_2$ be the subgraph of $G$ obtained by removing all vertices and edges
intersecting the open disc bounded by $c$ (so neither $G_1$ nor $G_2$ contains the edges of $G$ that cross $Q$).
We say that $G_1$ and $G_2$ are the \emph{$Q$-components} of $G$ (note that since $G$ is $2$-connected, the $Q$-components
are uniquely determined, independently of the exact choice of the curve $c$).

\section{Near-planar graphs}\label{sec-main}

We now aim to show that graphs drawn in the plane with crossings far apart are $5$-choosable.  For the purposes of
the induction, it will be useful to allow other kinds of irregularities (adjacent vertices with list of size three, as well as vertices
with list of size four not incident with the outer face, which arise when some vertices incident with a crossing are colored
and their colors are removed from the lists of their neighbors), subject to distance constraints.
This results in a rather technical statement, and we devote this section to its exact formulation.

Let us fix a drawing of a graph $G$ in the plane, possibly with crossings.
Let $P$ be a path of length at most three contained in the boundary of the outer face $F$ of
$G$ (where in particular, no edge of $P$ is crossed), let $N$ be a subset of $V(G)$ and $M$ a subset of $E(G)$, and let $L$ be a list assignment
for $G$ (generally, the set $N$ will contain vertices with list of size $4$, while the edges of $M$ will join vertices with lists of size three).
We say that a $5$-tuple $(G,P,N,M,L)$ is a \emph{target}.

\begin{figure}
\begin{center}
\includegraphics[width=110mm]{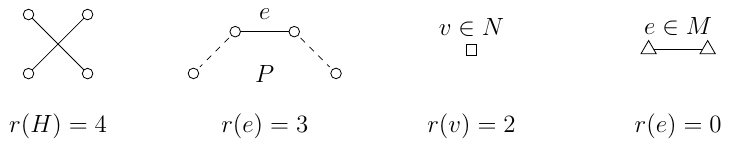}
\end{center}
\caption{Special subgraphs and their ranks.}
\label{fig-special}
\end{figure}
For a target $(G,P,N,M,L)$, we define some subgraphs $H$ of $G$ to be {\em special}, and assign a {\em rank $r(H)$} to each such subgraph (see Figure~\ref{fig-special}).  Specifically, $H$ is {\em special\/} if it falls into one of the following four cases:
\begin{itemize}
\setlength{\itemsep}{0pt}
\item $H$ consists of the two edges incident with a crossing. In this case, its rank is $4$.
\item $P$ has length three and $H$ consists of the middle edge of $P$; the rank of $H$ is $3$.
\item $H$ is equal to a vertex of $N$, and $r(H)=2$.
\item $H$ is equal to an edge of $M$, and $r(H)=0$.
\end{itemize}
For two subgraphs $H_1,H_2\subseteq G$, the {\em distance $d(H_1,H_2)$}
between $H_1$ and $H_2$ is the minimum of the distances between the vertices of $H_1$ and $H_2$, i.e., the minimum $k$
such that there exists a path $v_0v_1\ldots v_k$ in $G$ with $v_0\in V(H_1)$ and $v_k\in V(H_2)$.

\begin{figure}
\begin{center}
\includegraphics[width=120mm]{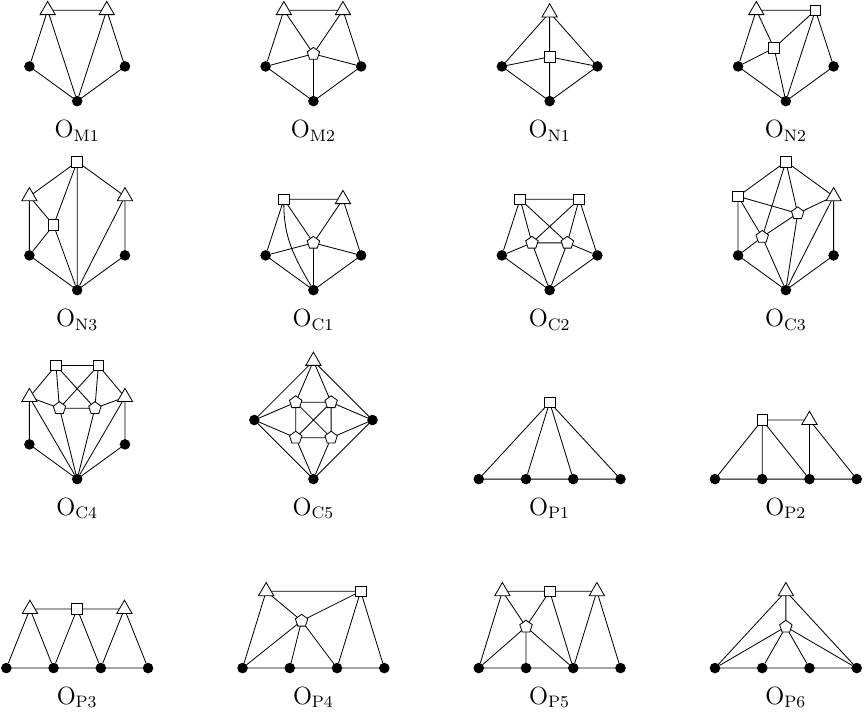}
\end{center}
\caption{The obstructions used in Theorem~\ref{thm-maingen}.}
\label{fig-obst}
\end{figure}

Given a target $(G,P,N,M,L)$, a subgraph $O\subseteq G$ is an {\em obstruction} if $O$ (with its drawing induced from the drawing of $G$) is isomorphic to one of the graphs drawn in
Figure~\ref{fig-obst} and sizes of the lists of its vertices match those prescribed by the figure, where the full-circle vertices have list of size one,
triangle vertices have list of size three, square vertices have list of size four and pentagonal vertices have list of size five.
Note that each obstruction contains a special subgraph of $G$.

We say that the target $(G, P, N, M, L)$ is \emph{valid} if the following conditions are satisfied (with $F$ denoting the outer face of $G$):
\begin{itemize}
\item[(S)] $|L(v)|\ge 5$ for $v\in V(G)\setminus (V(F)\cup N)$, $|L(v)|\ge 3$ for $v\in V(F)\setminus V(P)$ and $|L(v)|=1$ for $v\in V(P)$,
\item[(N)] $|L(v)|\ge 4$ for $v\in N\setminus V(F)$,
\item[(M)] if $|L(u)|=|L(v)|=3$ and $u$ and $v$ are adjacent, then $uv\in M$,
\item[(P)] $L$ gives a proper coloring to the subgraph induced by $V(P)$,
\item[(T)] if a vertex $v$ has three neighbors $w_1,w_2,w_3$ in $V(P)$, then $L(v)\neq L(w_1)\cup L(w_2)\cup L(w_3)$,
\item[(C)] if $x$ is a crossing and $G_x$ contains a vertex with list of size three, then all other vertices of $G_x$ have lists of size $1$ or $\ge5$,
\item[(D)] $d(H_1,H_2)\ge r(H_1)+r(H_2)+7$ for every pair $H_1\neq H_2$ of special subgraphs of $G$, and
\item[(O)] if the target contains an obstruction $O$, then $O$ is $L$-colorable.
\end{itemize}

We prove the following claim, which implies our main result, Theorem~\ref{thm-main0} (since each graph satisfying the assumptions
of Theorem~\ref{thm-main0} can be transformed into a legal target in a natural way, with $P$ consisting of a single vertex so that
the target contains no obstructions).

\begin{theorem}\label{thm-maingen}
If $(G, P, N, M, L)$ is a valid target, then $G$ is $L$-colorable.
\end{theorem}

We sometimes refer to (D) as the {\em distance condition}.
The purpose of the introduced rank function is the following. In our inductive arguments, we will occasionally
construct a smaller graph $G'$ and introduce a new special subgraph $H'$ in a vicinity of a special
subgraph $H$ that would no longer exist in $G'$. If $H'$ has smaller rank than $H$, then the distance condition
for special subgraphs in $G'$ would still hold, and the induction hypothesis can be applied.

Let us remark that if the distance condition holds, then $G$ can contain at most one of the obstructions.
For further reference we exhibit in Figure \ref{fig-badlists} all possible list assignments for which the obstructions are not colorable.
In particular, observe that the following holds:

\claim{cl-colprop}{Let $H$ be one of the obstructions and let $Q$ be the path in $H$ consisting of full-circle vertices.
Suppose that $Q$ has length two and that $H$ is neither\/ \ob{M1} nor\/ \ob{C1}.  Let $q$ be the middle vertex of $Q$ and let $L$ be a list assignment
such that each vertex $v$ drawn by a $k$-gon has $|L(v)|=k$, while the vertices of $Q$ have lists consisting of all possible colors.
Then there exists a color $b$ such that every $L$-coloring $\psi$ of $Q$ with $\psi(q)\neq b$ extends to an $L$-coloring of~$H$.}

\begin{figure}
\begin{center}
\includegraphics[width=140mm]{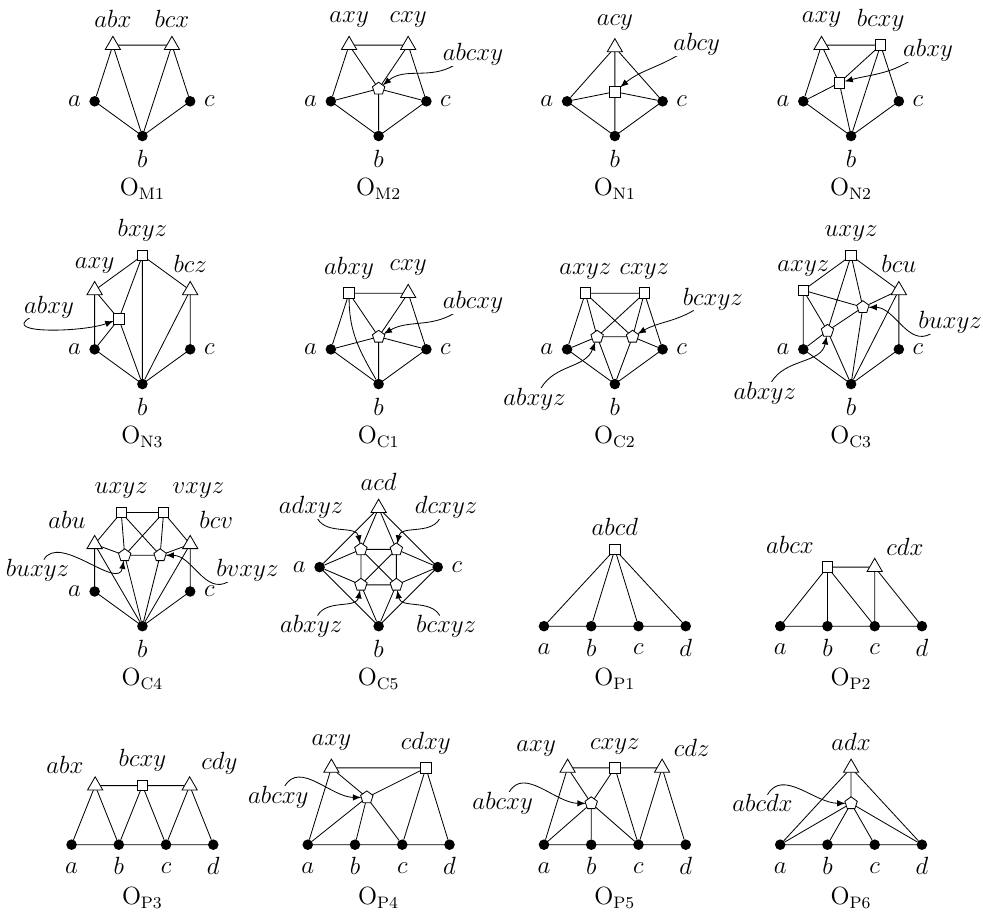}
\end{center}
\caption{The lists for which the obstructions cannot be colored. Colors represented by different letters may be equal to each other if they
do not occur in the same list for a particular obstruction.}
\label{fig-badlists}
\end{figure}

Let us now give a quick outline of the proof of Theorem~\ref{thm-maingen}; we work out all the details in the following section.
Essentially, we follow the proof of Theorem~\ref{thm-basic}.  First,
we show that the outer face of a hypothetical minimal counterexample $G$ has no chords and then we also restrict its $2$-chords.  This is somewhat
more complicated due to the presence of crossings and the condition (O).  Next, we find the set $X$ and its partial coloring $\vf$ defined in the same way as in the proof of Theorem~\ref{thm-basic},
and use it to construct the graph $G'$ with the list assignment $L'$.  By the minimality of $G$, we conclude that $G'$ violates
one of the assumptions of the theorem.  A straightforward case analysis shows that (O) holds, the definition of the rank function
ensures that (D) holds, and the conditions (S), (P) and (T)
follow in the same way as in the proof of Theorem~\ref{thm-basic}; but (M), (N) and (C) can be violated in ways which
do not enable us to obtain a contradiction directly.  However, we observe that in such a case, there is a special subgraph $S$
near to $X$.  In this situation, we apply the symmetric argument on the other side of the path $P$, and obtain another set $X'$
and a special subgraph $S'$ close to it.  By the distance condition, we have $S=S'$, and thus there exists a short
path from $X$ to $X'$ passing through $S$.  In this situation, we consider all the possible combinations of $X$, $X'$
and their positions relatively to $S$, and obtain a contradiction similarly to the way we deal with $2$-chords.

Let us note that the assumption (C) is a product of a somewhat delicate tradeoff.  We believe the claim still essentially
holds even without this assumption, and avoiding it would greatly reduce the number of possible bad cases and simplify the
last part of the proof.  However, the list of obstructions in (O) would be significantly larger, making the first part of
the proof longer and more complicated.  Moreover, if we omit (C) completely, then there exists an obstruction
with a precolored path of length one (see Figure~\ref{fig-needc}(a)), which would be a major problem (we could not easily get rid
of chords of $F$).
One could consider excluding Figure~\ref{fig-needc}(a) by forbidding vertices with lists of sizes three or four joined by a crossed edge.
This would still simplify the last part of the proof a lot.  However, in addition to having more than 10 new obstructions,
we do not see a way how to reduce the $2$-chord depicted in Figure~\ref{fig-needc}(b), which would need to be dealt with somehow.

\begin{figure}
\begin{center}
\includegraphics[width=70mm]{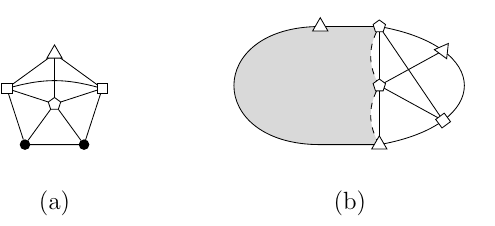}
\end{center}
\caption{Why is condition (C) needed?}
\label{fig-needc}
\end{figure}

Before we proceed with the proof of Theorem~\ref{thm-maingen}, let us show how it implies Theorem~\ref{thm-2x}.

\begin{proof}[Proof of Theorem~\ref{thm-2x}]
Let $G$ be a graph with crossing number at most two. We may assume that $G$ is nonplanar.
Consider a drawing of $G$ in the plane with one or two crossings and let $L$ be a list assignment such that each vertex has five admissible colors.
Let $xy$ and $uv$ be two edges crossing each other at the crossing $q$. Suppose first that the edges $xy$ and $uv$ do not participate in another crossing.
Now remove the two edges and add the edges $xu, uy, yv$, and $vx$ (if they are not already present). This gives rise to a graph $G'$ with at most one crossing,
and we can redraw it so that the cycle $xuyv$ bounds the outer face.  Let $L'$ be a list assignment such that $L'(x)\subset L(x)$, $L'(y)\subset L(y)$, and $L'(u)\subset L(u)$
are pairwise distinct sets of size $1$, $L'(v)=L(v)\setminus L'(u)$, and $L'(z)=L(z)$ for $z\in V(G)\setminus \{u,v,x,y\}$.
Theorem \ref{thm-maingen} implies that $G'$ has an $L'$-coloring $\vf$, and the choice of $L'$ ensures that $\vf$ is also an $L$-coloring of $G$.

If the edge $uv$ participates in another crossing, then $xy$ does not participate in another one.
Suppose that the segment of $uv$ from $u$ to the crossing $q$ does not contain the other crossing. Then we proceed similarly as above:
we remove the edges $xy$ and $uv$ and add edges $xu$ and $uy$. The resulting graph $G'$ is planar and the path $P=xuy$ is part of a facial walk.
Let $L'$ be defined as before; by Theorem \ref{thm-maingen} with $N=\{v\}$, $G'$ is $L'$-colorable, and thus $G$ is $L$-colorable.
\end{proof}

\section{The proof of Theorem~\ref{thm-maingen}}\label{sec-maingen}

We say that a target $(G,P,N,M,L)$ is a \emph{counterexample} if it is valid and $G$ is not $L$-colorable.
We proceed by contradiction, showing that no counterexample exists.
A counterexample $(G,P,N,M,L)$ is \emph{minimal} if there exists no counterexample $(G',P',N',M',L')$ satisfying
\begin{itemize}
\item $|V(G')|+|E(G')|<|V(G)|+|E(G)|$, or
\item $|V(G')|+|E(G')|=|V(G)|+|E(G)|$ and $\sum_{v\in V(G')} |L'(v)|<\sum_{v\in V(G)} |L(v)|$.
\end{itemize}

We follow the outline of the proof of Theorem~\ref{thm-basic}.  Let us first start by establishing some basic properties of
hypothetical minimal counterexamples.

\begin{lemma}\label{lemma-cl-basic}
Every minimal counterexample $B=(G,P,N,M,L)$ with the outer face $F$ has the following properties.
\begin{itemize}
\setlength{\itemsep}{0pt}
\item[\rm (a)] Every vertex $v\in V(G)$ satisfies $\deg(v)\ge |L(v)|$.
\item[\rm (b)] $G$ is $2$-connected and $\ell(P)\ge1$.
\item[\rm (c)] Every non-crossed chord of $F$ is incident with exactly one vertex of $P$, and this vertex is internal in $P$.
\item[\rm (d)] If $K$ is a triangle in $G$ and no edge of $K$ is crossed, then $K$ is not separating. If $K$ is a separating $4$-cycle without crossed edges,
then $\cin_K(G)-V(K)$ is either a vertex in $N$ or a complete graph on $4$ vertices
involving a crossing.
\item[\rm (e)] Every vertex $v\in V(G)$ satisfies $|L(v)|\le5$.
\item[\rm (f)] If $v\in V(G)\setminus V(P)$ is adjacent to a vertex $p\in V(P)$, then $L(p)\subseteq L(v)$.
\end{itemize}
\end{lemma}
\begin{proof}
For (a), if $v\in V(G)$ has degree less than $|L(v)|$, then every $L$-coloring of $G-v$ extends to an $L$-coloring of $G$, and thus the target obtained
from $B$ by removing $v$ is a counterexample contradicting the minimality of $B$.

For (f), it suffices to note that if $vp\in E(G)\setminus E(P)$ and $L(p)\cap L(v)=\emptyset$, then removing $e$ would result in a smaller counterexample.
Together with the property (P), this also implies that no chord of $F$ joins two vertices of $P$.

Hence, for (c), it suffices to consider a non-crossed chord $uv$ of $G$ that is not incident with an internal vertex of $P$.  Let $G_1$ and $G_2$ be the $uv$-components
of $G$ such that $P\subseteq G_1$.   By the minimality of $B$, there exists an $L$-coloring $\vf$ of $G_1$.  Let $L'(u)=\{\vf(u)\}$, $L'(v)=\{\vf(v)\}$
and $L'(x)=L(x)$ for $x\in V(G_2)\setminus \{u,v\}$.  Since $G$ is not $L$-colorable, it follows that $G_2$ is not $L'$-colorable, and thus
$(G_2,uv,N\cap V(G_2), M\cap E(G_2), L')$ is a counterexample contradicting the minimality of $B$.

The claim (b) and the first part of claim (d) (concerning triangles) is proved similarly to the previous paragraph.
For the second part of (d), let $K=v_1v_2v_3v_4$ be a separating $4$-cycle without crossed edges.  If all vertices of $K$
have a common neighbor $x\in N\cap V(\cin_K(G))$, then by the first part of (d), $\cin_K(G)-V(K)$ consists only of $x$.
Otherwise, we can choose the labels of the vertices of $K$ so that $v_4$ is not adjacent to a vertex in $N\cap V(\cin_K(G))$.  Furthermore, using (D),
we can choose the labels so that $v_4$ is not incident with a crossed edge other than possibly the chord $v_4v_2$.
By the minimality of $B$, there exists an $L$-coloring $\vf$ of the subgraph of $G$ consisting of $\cex_K(G)$ and all chords of $K$.
Consider the graph $G'=\cin_K(G)-v_4$ with the list assignment $L'$ given by $L'(v_i)=\{\vf(v_i)\}$ for $i=1,2,3$, $L'(x)=L(x)\setminus \{\vf(v_4)\}$
for all neighbors $x$ of $v_4$ distinct from $v_1$, $v_2$, and $v_3$, and $L'(x)=L(x)$ for all other vertices $x$ of $G'$.  Since $G$ is not $L$-colorable,
it follows that $G'$ is not $L'$-colorable, and by the minimality of $B$, the target $B'=(G', v_1v_2v_3, N\cap V(G'), \emptyset, L')$ is not
valid.  By the choice of $v_4$, we have $|L'(x)|\ge 4$ for all $x\in V(G')\setminus\{v_1,v_2,v_3\}$, and thus the target $B'$ may only violate (O);
and if it violates (O), then $B'$ contains the obstruction \ob{C2}.  However, together with the first part of (d), this implies that $\cin_K(G)-V(K)$ is a complete graph on $4$ vertices
involving a crossing.

Property (e) is achieved by removing colors from lists of size 6 or more. The only problem that may arise is that we obtain an obstruction; however,
inspection of bad lists for the obstructions exhibited in Figure \ref{fig-badlists} shows that we can always remove one of the colors so that (O) still holds.
\end{proof}

Next, we strengthen Lemma~\ref{lemma-cl-basic}(d) for triangles.

\begin{lemma}\label{lemma-cl-noseptr}
Let $B=(G,P,N,M,L)$ be a minimal counterexample and let $T$ be a triangle in $G$ (possibly with a crossed edge).
If $T$ does not bound the outer face of $G$, then $V(\cin_T(G))=V(T)$.
\end{lemma}
\begin{proof}
Let $T=v_1v_2v_3$. The claim follows by Lemma~\ref{lemma-cl-basic}(d) if no edge of $T$ is crossed; hence, suppose that say $v_1v_2$
is crossed by an edge $uw$, where the part of the edge $uw$ next to $w$ is drawn in the open disk bounded by $T$.
Let $G_1=\cex_T(G)+uw$. Suppose for a contradiction that $V(\cin_T(G))\neq V(T)$, and thus $G_1\neq G$ (since $\deg(w)>1$ if $w\neq v_3$ by Lemma~\ref{lemma-cl-basic}(a)).
By the minimality of $B$, there exists an $L$-coloring $\vf$ of $G_1$.
Let $G_2=\cin_T(G)$ and let $L_2$ be the list assignment such that $L_2(v_i)=\{\vf(v_i)\}$ for $i=1,2,3$,
$L_2(w)=L(w)\setminus \{\vf(u)\}$ if $w\neq v_3$, and let $L_2(z)=L(z)$ for $z\in V(G_2)\setminus \{v_1,v_2,v_3,w\}$.
Since $G$ is not $L$-colorable, $G_2$ is not $L_2$-colorable.
Let $N_2=N\cap V(G_2)$ if $w=v_3$, and $N_2=(N\cap V(G_2))\cup \{w\}$ otherwise.
Then $(G_2,v_1v_2v_3, N_2, \emptyset,L_2)$ is a counterexample contradicting the minimality of $B$.
\end{proof}

We get the following easy corollary.

\begin{lemma}\label{lemma-cl-noobstr}
No minimal counterexample $(G,P,N,M,L)$ contains an obstruction.
\end{lemma}
\begin{proof}
Suppose for a contradiction that $H\subseteq G$ is an obstruction.  Note that $H$ contains a special subgraph, and by (D), if $e\in E(H)$ is not crossed in $H$,
then $e$ is not crossed in $G$, either.  Furthermore, no vertex of the outer face of $H$ belongs to $N$, and thus all of them are incident
with the outer face of $G$. Lemmas~\ref{lemma-cl-basic}(c) and \ref{lemma-cl-noseptr} imply that $G=H$.  However, then $G$ is $L$-colorable by (O),
which is a contradiction.
\end{proof}

Analogously, we can strengthen Lemma~\ref{lemma-cl-basic}(d) for $4$-cycles.

\begin{lemma}\label{lemma-cl-nosepq}
Let $B=(G,P,N,M,L)$ be a minimal counterexample and let $K$ be a $4$-cycle in $G$ (possibly with a crossed edge).
If $K$ does not bound the outer face of $G$ and $V(\cin_K(G))\neq V(K)$, then one of the following holds:
\begin{itemize}
\item $\cin_K(G)-V(K)$ is isomorphic to $K_4$ drawn with a crossing, or
\item $V(\cin_K(G))=V(K)\cup\{z\}$ for a vertex $z$ adjacent to all vertices of $K$ such that either $z\in N$ or
$z$ is incident with an edge crossing an edge of $K$.
\end{itemize}
\end{lemma}
\begin{proof}
Let $K=v_1v_2v_3v_4$.  If no edge of $K$ is crossed, then the claim follows from Lemma~\ref{lemma-cl-basic}(d);
hence, suppose that say $v_3v_4$ is crossed by an edge $uz$, where the part of the edge $uz$ next to $z$ is drawn in the open disk bounded by $K$.
If $\cin_K(G)$ contained a chord of $K$, then $V(\cin_K(G))=V(K)$ by Lemma~\ref{lemma-cl-noseptr}; hence, we can assume that $K$ is an induced
cycle in $\cin_K(G)$.

Let $G_1=\cex_K(G)+uz$. Suppose that $V(\cin_K(G))\neq V(K)$, and thus $G_1\neq G$ (since $\deg(z)>1$ if $z\not\in V(K)$ by Lemma~\ref{lemma-cl-basic}(a)).
By the minimality of $B$, there exists an $L$-coloring $\vf$ of $G_1$.
Consider the graph $G'=\cin_K(G)-v_4$ with the list assignment $L'$ given by $L'(v_i)=\{\vf(v_i)\}$ for $i=1,2,3$,  $L'(z)=L(z)\setminus \{\vf(u),\vf(v_4)\}$
if $z\not\in \{v_1,v_2\}$ and $zv_4\in E(G)$, $L'(z)=L(z)\setminus \{\vf(u)\}$ if $z\not\in \{v_1,v_2\}$ and $zv_4\not\in E(G)$,
$L'(x)=L(x)\setminus \{\vf(v_4)\}$ for all neighbors $x$ of $v_4$
distinct from $v_1$, $v_2$, $v_3$, and $z$, and $L'(x)=L(x)$ for all other vertices $x$ of $G'$.
Since $G$ is not $L$-colorable, it follows that $G'$ is not $L'$-colorable, and by the minimality of $B$, the target $B'=(G', v_1v_2v_3, N\cap V(G'), \emptyset, L')$ is not
valid.  Since the edge $v_3v_4$ is crossed, the distance condition for $B$ implies that $B'$ does not contain any obstruction, and thus it satisfies (O).
Observe that $B'$ satisfies all other conditions of validity except possibly for (T).  If $B'$ violates (T), then $z\not\in V(K)$ and $z$ is adjacent to all vertices of $K$,
and by Lemma~\ref{lemma-cl-noseptr}, we have $V(\cin_K(G))=V(K)\cup\{z\}$.
\end{proof}

Next, let us consider separating $5$-cycles, in a somewhat restricted situation.

\begin{lemma}\label{lemma-cl-noseppen}
Let $B=(G,P,N,M,L)$ be a minimal counterexample and let $K$ be a $5$-cycle in $G$ such that no edge of $K$ is crossed,
$K$ does not bound the outer face of $G$, and $V(\cin_K(G))\neq V(K)$.  If $G$ contains a special subgraph $S$ such that $S\subseteq \cex_K(G)$
and $d(S,K)\le 1$, then $V(\cin_K(G))=V(K)\cup\{z\}$ for a vertex $z$ adjacent to all vertices of $K$.
\end{lemma}
\begin{proof}
Observe that $K=v_1v_2v_3v_4v_5$ is an induced cycle in $\cin_K(G)$, as otherwise Lemma~\ref{lemma-cl-nosepq} implies that
$\cin_K(G)$ contains a special subgraph at distance at most $1$ from $K$, which together with $S$ violates the assumption (D) for $B$.
By the minimality of $B$, there exists an $L$-coloring $\vf$ of $\cex_K(G)$.  Let $G'=\cin_K(G)-\{v_4,v_5\}$ with the list assignment $L'$ given by
$L'(v_i)=\{\vf(v_i)\}$ for $i=1,2,3$ and $L'(x)=L(x)\setminus\{\vf(v_i):i\in\{1,2\}, xv_i\in E(G)\}$ for $x\in V(G')\setminus\{v_1,v_2,v_3\}$.  Since $G$ is not $L$-colorable,
it follows that $G'$ is not $L'$-colorable, and by the minimality of $B$, the target $B'=(G', v_1v_2v_3, N\cap V(G'), \emptyset, L')$ is not
valid.  Note that by the distance condition for $S$ and by Lemma~\ref{lemma-cl-noseptr}, at most one vertex $z$ of $G'$ satisfies $|L'(z)|=3$ and $B'$ satisfies (O).
Observe that $B'$ satisfies all other conditions of validity except possibly for (T).  If $B'$ violates (T), then $z$ is adjacent to all vertices of $K$, and
by Lemma~\ref{lemma-cl-noseptr}, we have $V(\cin_K(G))=V(K)\cup\{z\}$.
\end{proof}

Our next goal is to show that the outer face of a minimal counterexample does not have chords.
Let us start with excluding non-crossed chords.

\begin{lemma}\label{lemma-cl-nochord-nocr}
If $B=(G,P,N,M,L)$ is a minimal counterexample and $F$ is the outer face of $G$, then every chord of $F$ is crossed.
\end{lemma}
\begin{proof}
Let $k=\ell(P)$, and let $P=p_0p_1\ldots p_k$. 
Suppose for a contradiction that $F$ has a non-crossed chord $uv$. By Lemma~\ref{lemma-cl-basic}(c), $u$ is an internal vertex
of $P$, say $u=p_1$, while $v\not\in V(P)$.  Let $G_1$ and $G_2$ be the $uv$-components of $G$ such that $p_0\in V(G_1)$. Let $P_1=p_0p_1v$ and
$P_2=vp_1\ldots p_k$. For each color $c\in L(v)\setminus L(u)$, let $L_c$ be the list assignment such that $L_c(v)=\{c\}$ and $L_c(z)=L(z)$ if $z\neq
v$.  Since $G$ is not $L$-colorable, either $G_1$ or $G_2$ is not $L_c$-colorable.  Furthermore, since both $G_1$ and $G_2$ are $L$-colorable (by the
minimality of $B$), there exist distinct colors $c_1$ and $c_2$ such that $G_1$ is not $L_{c_1}$-colorable and $G_2$ is not $L_{c_2}$-colorable.
For $i\in\{1,2\}$, let $B_i=(G_i,P_i,N\cap V(G_i), M\cap E(G_i), L_{c_i})$.
By the minimality of $B$, neither $B_1$ nor $B_2$ is a valid target.
Clearly, both $B_1$ and $B_2$ satisfy (S), (N), (M), (C), and (D).  Hence, each of them violates (P), (T), or (O).

We claim that we can choose the labeling of $P$ and the chord $uv=p_1v$ so that with the notation as in the previous paragraph,
the following holds.
\begin{itemize}
\item[(i)] $B_1$ does not contain any obstruction,
\item[(ii)] no edge at distance at most two from $p_1$ in $G_1$ is crossed, and
\item[(iii)] $p_0$, $p_1$, and $v$ do not have a common neighbor $z$ with $|L(z)|=3$.
\end{itemize}
Indeed, if $B_1$ contains an obstruction, then it contains a special subgraph whose distance to $p_1$ is at
most two.  In that case, (D) implies that $\ell(P)=2$ and that $B_2$ contains no obstruction.  Hence, we can exchange
the labels of $p_0$ with $p_2$, $G_1$ with $G_2$, \ldots, and $B_1$ with $B_2$, so that (i) holds.
Similarly, if an edge at distance at most two from $p_1$ in $G_1$ is crossed, then (D) implies that $\ell(P)=2$, $B_2$ contains no obstruction,
and no edge at distance at most two from $p_1$ in $G_2$ is crossed, and thus exchanging the labels in the same way ensures that both (i) and (ii) hold.
Finally, if $p_0$, $p_1$, and $v$ have a common neighbor $z\in V(G_1)$ with list of size three, then $z$ is incident with $F$ and by
(ii) and by Lemma~\ref{lemma-cl-basic}(c) and (d), we conclude that $V(G_1)=\{p_0,p_1,v,z\}$; in this case, we consider the chord $p_1z$ instead of the chord $p_1v$
(so that $G_1$ becomes equal to the triangle $p_0p_1v$), which ensures that (i), (ii), and (iii) hold.

By (i) and (iii), $B_1$ satisfies (O) and (T). Therefore, $B_1$ violates (P), and thus $p_0v\in E(G)$. By (ii) and by Lemma~\ref{lemma-cl-basic}(c) and (d),
we conclude that $G_1$ is equal to the triangle $p_0p_1v$.

Let $S=L(v)\setminus\bigcup_{p\in V(P), vp\in E(G)} L(p)$.
Recall that by Lemma~\ref{lemma-cl-noobstr}, $B$ does not contain the obstruction \ob{P1}. Together with the property (T) of $B$, this implies that $S\neq\emptyset$.
Consider any $c\in S$, and let $B_c=(G_2,P_2,N\cap V(G_2),M\cap E(G_2), L_c)$.  Since $G_1$ is $L_c$-colorable, we conclude that $G_2$ is not $L_c$-colorable,
and thus $B_c$ is not a valid target.  Note that $B_c$ satisfies (S), (N), (M), (C), and (D), and that $B_c$ satisfies (P) since $c\in S$.

If $B_c$ contained an obstruction $H$, then Lemmas~\ref{lemma-cl-basic}(c) and \ref{lemma-cl-noseptr} would imply that $G_2=H$.
However, the inspection of the obstructions shows that either $G=H\cup G_1$ would be $L$-colorable, or $B$ would contain an obstruction.
Therefore, $B_c$ satisfies (O).

It follows that $B_c$ violates (T), i.e., there exists a vertex $w$ adjacent to $v$ and to vertices $p,p'\in V(P)\setminus \{p_0\}$ such that
$L(w)=L(p)\cup L(p')\cup \{c\}$.  Since this holds for all $c\in S$, it follows that $|S|=1$.  If $|L(v)|=3$, then $vw\in M$, (D) for $B$ implies that
$\ell(P)=2$, and $B$ contains \ob{M1}, which contradicts Lemma~\ref{lemma-cl-noobstr}.  Hence, suppose that $|L(v)|\ge 4$, and since $|S|=1$,
$v$ has at least three neighbors in $P$.  If $\ell(P)=2$, then $|L(v)|=4$ and the edges $vp_2$ and $wp_1$ cross, which contradicts (C) for $G$.
Hence, $\ell(P)=3$, and by (D) for $B$, neither $v$ nor $w$ is incident with a crossed edge.  Therefore, $v$ is adjacent to $p_0$, $p_1$, and $p_2$,
$|L(v)|=4$, and $w$ is adjacent to $p_2$ and $p_3$.  However, this implies that $G$ contains \ob{P2}, which contradicts Lemma~\ref{lemma-cl-noobstr}.
\end{proof}

Next, we exclude crossed chords that are not incident with an internal vertex of $P$.

\begin{lemma}\label{lemma-cl-nochord-p}
If $B=(G,P,N,M,L)$ is a minimal counterexample and $F$ is the outer face of $G$, then every chord of $F$ is incident with an internal vertex of $P$.
\end{lemma}
\begin{proof}
Suppose for a contradiction that $F$ has a chord $uv$ of $F$ that is not incident with an internal vertex of $P$.
By Lemma~\ref{lemma-cl-nochord-nocr}, there exists an edge $e$ crossing the edge $uv$.
Let $G_1$ and $G_2$ be the $uv$-components of $G$ such that $P\subseteq G_1$.  Let $e=x_1x_2$, where $x_1\in V(G_1)$ and $x_2\in V(G_2)$.
By the minimality of $G$, there exists an $L$-coloring $\vf$ of $G_1$.  Since $\vf(u)\neq\vf(v)$, we can by symmetry assume that $\vf(x_1)\neq \vf(u)$.

Let $G'$ be the graph obtained from $G_2-uv$ by adding new vertices $y_1$ and $y_2$, edges of the path $P'=uy_1y_2v$ and the edge $y_1x_2$.
Let $L'$ be the list assignment for $G'$ such that $L'(u)=\{\vf(u)\}$, $L'(v)=\{\vf(v)\}$, $L'(y_1)=\{\vf(x_1)\}$, $L'(y_2)=\{c\}$ for a new color
$c$ that does not appear in any of the lists and $L'(z)=L(z)$ for any other vertex $z$.
Let $B'=(G',P',N\cap V(G_2), M\cap E(G_2), L')$.  

Note that the edge $y_1y_2$ is a special subgraph in $B'$ which does not appear in $B$.  However, $B'$ satisfies (D), since the crossing of $G$ incident with
$x_2$ does not belong to $G'$ and any path from a special subgraph in $G'$ to $y_1y_2$ passes through one of the vertices $u,v,x_2$ of the crossing in $G$.
Furthermore, $B'$ satisfies (O) since $y_2$ has degree two in $G'$.
Similarly, $B'$ satisfies all other conditions of validity except possibly for (T).

Since $G$ is not $L$-colorable, $G'$ is not $L'$-colorable, and thus $B'$ is not a valid target.  Hence, $B'$ violates (T).
This implies that $x_2$ has list of size three and it is adjacent to $u$ and $v$.  By Lemmas~\ref{lemma-cl-basic}(c) and \ref{lemma-cl-noseptr},
we have $V(G_2)=\{u,v,x_2\}$.  Note that by (C), we conclude that each of
$|L(u)|$, $|L(v)|$, and $|L(x_1)|$ is either 1 or 5.  Let $a$ be a color in $L(x_2)$ distinct from the colors of its neighbors in $P$, which exists by
(T).  Let $G''=G-x_2$ with the list assignment $L''$ such that $L''(z)=L(z)\setminus\{a\}$ for $z\in\{u,v,x_1\}$ and $L''(z)=L(z)$
otherwise.  Let $B''=(G'',P,N\cup \{x_1\}, M,L'')$, and observe that $G''$ is not $L''$-colorable and that $B''$ satisfies all the conditions of validity
except possibly for (O).

By the minimality of $G$, it follows that $B''$ violates (O).  Let $H$ be an obstruction in $B''$.  By Lemma~\ref{lemma-cl-noobstr},
$H$ is not an obstruction in $B$, and thus $V(H)\cap \{u,v,x_1\}\neq\emptyset$.  Thus, the distance between the special subgraph of $H$ and between
$\{u,v,x_1\}$ is at most $2$, and since $u$, $v$, and $x_1$ are incident with a crossing in $G$ that does not appear in $H$, the condition (D)
for $B$ implies that $x_1$ with $|L''(x_1)|=4$ is a special subgraph of $H$ and that $\ell(P)=2$.  Consequently, $B''$ contains
one of \ob{N1}, \ob{N2} or \ob{N3}, in which the interior vertex with list of size 4 is $x_1$.
However, inspection of these graphs shows that $|L''(u)|=3$ or $|L''(v)|=3$, and thus $|L(u)|\le 4$ or $|L(v)|\le 4$, which contradicts (C).
\end{proof}

Finally, let us exclude all the remaining chords.

\begin{lemma}\label{lemma-cl-nochord}
The outer face of a minimal counterexample has no chords.
\end{lemma}
\begin{proof}
Let $B=(G,P,N,M,L)$ be a minimal counterexample and let $F$ be the outer face of $G$.
Suppose for a contradiction that $F$ has a chord $uv$.  By Lemmas~\ref{lemma-cl-nochord-nocr} and \ref{lemma-cl-nochord-p},
$uv$ is crossed and incident with an internal vertex of $P$.  By (D), we conclude that $\ell(P)=2$; let $P=p_0p_1p_2$, where say $u=p_1$.
Let $e$ be the edge crossing $uv$ and let $G_1$ and $G_2$ be the $uv$-components of $G-e$ such that $p_0\in V(G_1)$ and $p_2\in V(G_2)$.
Let $P_1=p_0p_1v$ and $P_2=p_2p_1v$, and let $e=x_1x_2$, where $x_i\in V(G_i)$ for $i\in \{1,2\}$.

If $G_i$ contains an edge $f$ different from $p_{2i-2}p_1$, $p_{2i-2}v$, and $p_1v$, then by the minimality of
$G$ there exists an $L$-coloring $\vf_{3-i}$ of $G-f\supseteq G_{3-i}+x_1x_2$.  If additionally $|L(x_i)|\in \{1,5\}$, then
define $L_i$ to be the list assignment for $G_i$ such that $L_i(v)=\{\vf_{3-i}(v)\}$, $L_i(x_i)=L(x_i)\setminus \{\vf_{3-i}(x_{3-i})\}$,
and $L_i(z)=L(z)$ for any other vertex $z$.  Let $B_i=(P_i, (N\cap V(G_i))\cup \{x_i\}, M\cap E(G_i), L_i)$, and note that
$B_i$ satisfies all the conditions of validity except possibly for (P), (T), or (O).
Observe that $G_i$ is not $L_i$-colorable, and by the minimality of $B$, we conclude that $B_i$ violates (P), (T) or (O).
Since $\vf_{3-i}$ is a coloring of $G-f$, (P) is satisfied by $B_i$.
Since $B$ satisfies (D) and all chords of $F$ are crossed by Lemma~\ref{lemma-cl-nochord-nocr}, it follows that $B_i$ satisfies (T).  Thus, $B_i$ violates (O).
The corresponding obstruction is \ob{N1} since all others either have a special subgraph that would violate the distance condition in $G$,
or have a non-crossed chord incident with $p_1$ which contradicts Lemma~\ref{lemma-cl-nochord-nocr}.

Hence, for $i\in\{1,2\}$, either $E(G_i)\subseteq \{p_{2i-2}p_1,p_{2i-2}v,p_1v\}$, or $|L(x_i)|\in \{3,4\}$, or $B_i$ as defined in the previous paragraph contains \ob{N1}.
Together with Lemmas~\ref{lemma-cl-noseptr} and \ref{lemma-cl-nochord-nocr}, we conclude that one of the following holds:
\begin{itemize}
\item $x_i\in V(P_i)$ and either $G_i=P_i$ or $G_i$ is the triangle on $V(P_i)$, or
\item $|L(x_i)|\in \{3,4\}$, or
\item $G_i$ is equal to \ob{N1} and $x_i$ is its vertex with list of size four.
\end{itemize}
By Lemma~\ref{lemma-cl-nochord-p}, at most one of $x_1$ and $x_2$ has a list of size three or four.
By symmetry, we can assume that $|L(x_1)|\in \{1,5\}$.

If $|L(x_2)|\in \{1,5\}$, then observe that all possible combinations of graphs $G_1$ and $G_2$ (each of which is a path, a triangle, or \ob{N1})
are either $L$-colorable or equal to \ob{C1}. Therefore, $|L(x_2)|\in \{3,4\}$, and thus $x_2$ is incident with $F$.
By Lemmas~\ref{lemma-cl-nochord-nocr} and \ref{lemma-cl-nochord-p}, we have $x_1\not\in V(F)$, and thus $G_1$ is \ob{N1}.

Let $w$ be the vertex of $G_1$ with list of size three, let $G'=G-\{w,p_0\}-p_1v$ and let $L'$ be the list assignment such that
$L'(x_1)=\{\vf_1(x_1)\}$, $L'(v)=\{\vf_1(v)\}$ and $L'(z)=L(z)$ otherwise.  Let $B'=(G',p_2p_1x_1v, N\cap V(G'), M\cap E(G'),L')$
and note that $G'$ is not $L'$-colorable and that the target $B'$ satisfies all the conditions of validity except possibly for (T) and (O).

If $v$ has degree at least $5$ in $G$, then it has degree at least three in $G'$.  Together with Lemmas~\ref{lemma-cl-noseptr} and \ref{lemma-cl-nochord-nocr}, this implies
that $x_2\in V(F)$ is not adjacent to $v$, hence $B'$ satisfies (T).  If $v$ has degree at most four, then $|L(v)|\le 4$ by Lemma~\ref{lemma-cl-basic}(a), and
by (C), $|L(x_2)|=4$, and again $B'$ satisfies (T).  By the minimality of $B$, we conclude that $B'$ violates (O).  Since $x_1$ has degree three in $G'$ and it is adjacent to
a vertex $x_2$ with list of size three or four, $G'$ contains (and by Lemmas~\ref{lemma-cl-noseptr} and \ref{lemma-cl-nochord-nocr}, is equal to) \ob{P1} or \ob{P2}.  However,
then $p_1x_2$ is a non-crossed chord of $F$, which contradicts Lemma~\ref{lemma-cl-nochord-nocr}.
\end{proof}

Let us show an easy corollary.

\begin{lemma}\label{lemma-cl-nopcr}
If $B=(G,P,N,M,L)$ is a minimal counterexample, then no vertex of $P$ is incident with a crossed edge.
\end{lemma}
\begin{proof}
Suppose for a contradiction that $vp$ is a crossed edge with $p\in V(P)$.  By Lemma~\ref{lemma-cl-nochord}, we have $v\not\in V(F)$.  Furthermore, since $P$
is incident with a crossing, (D) implies that $\ell(P)\le 2$.  Let $L'$ be the list
assignment such that $L'(v)=L(v)\setminus L(p)$ and $L'$ matches $L$ on the rest of the vertices of $G$.  Note that $G-vp$ is not $L'$-colorable,
and by the minimality of $B$, we conclude that $B'=(G-vp,P,N\cup\{v\},M,L')$ contains \ob{N1}, \ob{N2} or \ob{N3}, whose internal vertex with list of size 4 is $v$.
Note that $B'$ cannot contain \ob{N1}, since $v$ is not adjacent to all vertices of $P$ in $G-vp$.
Similarly, $B'$ cannot contain \ob{N3}, since the edge $vp$ would be crossed twice.  Finally, $B'$ does not contain
\ob{N2}, since $G$ does not contain \ob{C1} by Lemma~\ref{lemma-cl-noobstr}.
This is a contradiction.
\end{proof}

Next, we prove the following claim that is useful when considering the
condition (T) for targets derived from minimal counterexamples.

\begin{lemma}\label{lemma-cl-no3inp}
If $B=(G,P,N,M,L)$ is a minimal counterexample, then every vertex of $G$ has at most two neighbors in $P$.
\end{lemma}
\begin{proof}
Let $k=\ell(P)$ and $P=p_0p_1\ldots p_k$. Suppose for a contradiction that a vertex $v\in V(G)$
has at least three neighbors $p_a,p_b,p_c\in V(P)$, where $0\le a<b<c\le k$.  By Lemma~\ref{lemma-cl-nochord}, $v$ is not incident with the outer face $F$ of $G$.
Let $K$ be the cycle $p_ap_{a+1}\ldots p_cv$, and note that $K$
has a chord $vp_b$.  By Lemma~\ref{lemma-cl-nopcr}, none of the edges $vp_a$, $vp_b$ and $vp_c$ is crossed.  By Lemma~\ref{lemma-cl-basic}(d) applied to the two
cycles in $K+vp_b$ containing the edge $vp_b$, the cycle $K$ is not separating (Lemma~\ref{lemma-cl-basic}(d) allows a vertex of $N$ or a $K_4$ with a crossed edge
in the interior of $K$; however, this would only be possible if $\ell(P)=3$, yielding two special subgraphs at distance 1).

Suppose first that $c-a=\ell(P)$, and let $G_2$ be the $p_avp_c$-component of $G$ that does not contain $P$.
Since $v\notin V(F)$, and $v\notin N$ if $\ell(P)=3$, there exists a color $g\in L(v)$ that does not appear in the lists of vertices in $P$.
Let $L_2$ be the list obtained from $L$ by setting $L_2(v)=\{g\}$.
Observe that $G_2$ is not $L_2$-colorable, and thus $B_2=(G_2,p_avp_c,N,M,L_2)$ is not a valid target.
This is only possible if $B_2$ violates either (T) or (O).  In the former
case, $G$ is either \ob{N1} or \ob{P6}, which contradicts Lemma~\ref{lemma-cl-noobstr}.  In the latter case, we have $\ell(P)=2$ by the distance condition,
and \refclaim{cl-colprop} together with Lemmas~\ref{lemma-cl-noseptr} and \ref{lemma-cl-nochord-nocr} imply that $G$ is $L$-colorable
unless $G_2$ is either \ob{M1} or \ob{C1}.  If $G_2$ is \ob{M1}, then $G$ is \ob{M2},
and if $G_2$ is \ob{C1}, then $G$ is $L$-colorable; in both cases, we obtain a contradiction.

Hence, $c-a<\ell(P)$, and thus $\ell(P)=3$; by symmetry, we can assume that $v$ is adjacent to say $p_0$, $p_1$ and $p_2$, and $v$ is not adjacent to $p_3$.
If $L(p_0)=L(p_2)$, then $G-vp_2$ gives a counterexample contradicting the minimality of $B$.  Therefore, $L(p_0)\neq L(p_2)$.  Since the edges $vp_0$, $vp_1$, and $vp_2$
are not crossed by Lemma~\ref{lemma-cl-nopcr}, Lemma~\ref{lemma-cl-basic}(d) implies that the degree of $p_1$ is three.  Let $G'=G-p_1+p_0p_2$,
with the list assignment $L'$ such that $L'(v)=L(v)\setminus L(p_1)$ and $L'(z)=L(z)$ for $z\in V(G')\setminus \{v\}$.
Let $B'=(G,p_0p_2p_3, N\cup \{v\}, M,L')$.  Note that $B'$ satisfies (D), since the rank of the special subgraph
$p_1p_2$ in $B$ is greater than the rank of the special subgraph $v$ in $B'$, and any path $Q$ between
two special subgraphs $S_1$ and $S_2$ that uses the new edge $p_0p_2$ gives rise to paths between $S_1$ or $S_2$ and the middle edge $p_1p_2$ of $P$ in $G$, thus implying
$\ell(Q)\ge 14+r(S_1)+r(S_2)+2r(p_1p_2)-1>7+r(S_1)+r(S_2)$.  Since $G'$ is not $L'$-colorable, the minimality of $B$ implies that $B'$ is not valid,
and this is only possible if $B'$ violates (O).  Hence, $B'$ contains \ob{N1}, \ob{N2}, or \ob{N3}, with $v$ as the vertex with list of size $4$ not incident with the outer face.
However, then $B$ contains \ob{P6}, \ob{P4}, or \ob{P5}, respectively, which contradicts Lemma~\ref{lemma-cl-noobstr}.
\end{proof}

Let us now derive a variation on Lemma~\ref{lemma-cl-noseppen}.
\begin{lemma}\label{lemma-cl-nosepnearc}
Let $B=(G,P,N,M,L)$ be a minimal counterexample.  Suppose that $Q=x_1x_2\ldots x_{t-1}x_t$ is a path in $G$, where $t\le 6$ and $x_1x_2$ crosses $x_{t-1}x_t$.  Let $c$ be the closed
curve consisting of the path $x_2\ldots x_{t-1}$ and parts of the edges $x_1x_2$ and $x_{t-1}x_t$.
If $x_1$ is not drawn in the open disk $\Lambda$ bounded by $c$, then no vertex of $G$ is contained in $\Lambda$.
\end{lemma}
\begin{proof}
Observe first that the curve $c$ is not crossed since all its edges are close to a crossing. Let $G'$ be the subgraph of $G$ consisting of the vertices
and edges drawn fully in the closure of $\Lambda$.  Suppose for a contradiction that $x_1\not\in \Lambda$ (and hence $x_t\not\in\Lambda$) and that $\Lambda$ contains a vertex $v$ of $G$.
By the minimality of $B$, there exists an $L$-coloring $\vf$ of $G-v$.  Let $L'$ be the list assignment for $G'$ defined by $L'(x_i)=\{\vf(x_i)\}$ for $i=2,\ldots,t-1$
and by $L'(z)=L(z)$ for $z\in V(G)\cap \Lambda$.  Note that $G'$ is not $L'$-colorable.  However, then $B'=(G', x_2\ldots x_{t-1}, N\cap V(G'),\emptyset, L')$ is
a counterexample contradicting the minimality of $B$ (note that $B'$ satisfies (D) even if $t=6$, since the middle edge of the path $x_2x_3x_4x_5$ has
smaller rank then the crossing, whose distance to $x_3x_4$ in $G$ is one).
\end{proof}

Next, we consider an analogue of a chord formed by parts of two crossing edges.

\begin{lemma}\label{lemma-cl-crossnn}
Let $B=(G,P,N,M,L)$ be a minimal counterexample, and let $F$ denote the outer face of $G$.
Let $uv$ and $xy$ be crossing edges with $u,x\in V(F)$, and let $c$ denote the curve formed by the parts of $uv$ and $xy$ between
the crossing and $u$ and $x$, respectively.  If $c$ is not a part of the boundary of $F$, then the part of $G$ separated from $P$ by $c$ consists only of the edge $ux$.
\end{lemma}
\begin{proof}
By Lemma~\ref{lemma-cl-nopcr}, neither $u$ nor $x$ belongs to $P$, and by Lemma~\ref{lemma-cl-nochord}, we have $v,y\not\in V(F)$.
Suppose that $c$ is not part of the boundary of $F$.  Let $G_2$ be the subgraph of $G$ consisting of the vertices and edges that are fully
drawn inside the closed disc bounded by $c$ and the part of the boundary of $F$ between $u$ and $x$ that does not contain $P$.
Note that there are two possible situations, depending on whether $G_2$ includes the vertices $v$ and $y$ or not. In either case,
we can write $G=G_1\cup G_2$ for a subgraph $G_1$ of $G$ such that $G_1\cap G_2$
consists only of vertices $u$ and $x$. Let $G_2'$ be the graph obtained from $G_2$ by adding a common neighbor $w$
of $u$ and $x$.
By the minimality of $B$, there exists an $L$-coloring $\vf$ of $G_1$. Let $L_2$ be a list assignment such that $L_2(u)=\{\vf(u)\}$, $L_2(x)=\{\vf(x)\}$,
$L_2(w)=\{a\}$ for some color $a$ distinct from $\vf(u)$ and $\vf(v)$, and $L_2(z)=L(z)$ for $z\in V(G'_2)\setminus \{u,x,w\}$.
Let $B_2=(G'_2,uwx,N\cap V(G'_2), M\cap E(G'_2), L_2)$.  Note that $G'_2$ is not $L_2$-colorable, and by the minimality of $B$, it follows that $B_2$ is not
a valid target.  This is only possible if $B_2$ violates (P), and thus $ux\in E(G)$.

If $v,y\not\in V(G_2)$, then by Lemmas~\ref{lemma-cl-nochord} and \ref{lemma-cl-nosepnearc}, it follows that $G_2$ consists only of the edge $ux$.
If $v,y\in V(G_2)$, then we can redraw the edge $ux$ along $c$ so that $ux$ becomes a chord of $F$, which contradicts Lemma~\ref{lemma-cl-nochord}.
\end{proof}

\begin{figure}
\begin{center}
\includegraphics[width=100mm]{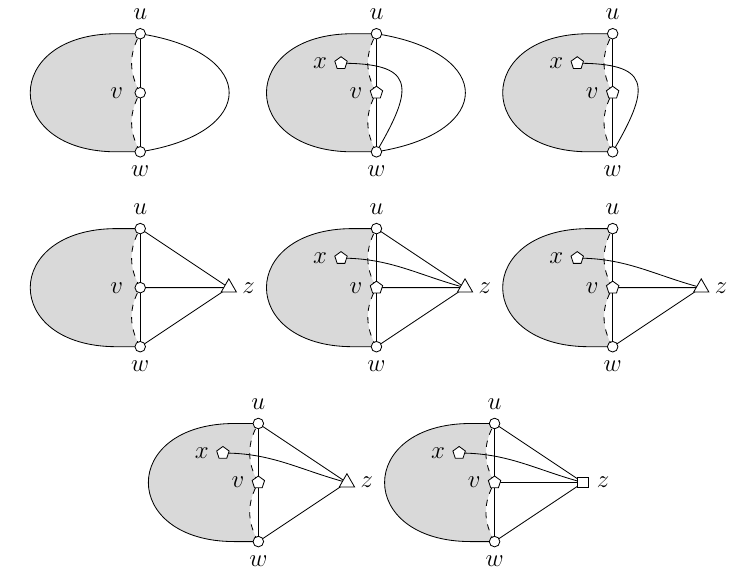}
\end{center}
\caption{Possible cases for $G_2$ for a $2$-chord $uvw$.}
\label{fig-2chord}
\end{figure}

Now, we shall consider the 2-chords of $F$.

\begin{lemma}\label{lemma-cl-2chord}
Let $B=(G,P,N,M,L)$ be a minimal counterexample, and let $F$ denote the outer face of $G$.
Let $uvw$ be a $2$-chord of $F$ such that $vw$ is not crossed and neither $u$ nor $w$ is an internal vertex of $P$.
If $G_1$ and $G_2$ are the $uvw$-components of $G$ with $P\subset G_1$, then one of the following holds (cf.\ Figure~\ref{fig-2chord}):
\begin{itemize}
\setlength{\itemsep}{0pt}
\item $V(G_2)=\{u,v,w\}$, and either $uv$ is not crossed and $uw\in E(G)$, or
$uv$ is crossed by an edge incident with $w$; in the latter case, $uw$ may or may not be an edge.
\item $V(G_2)=\{u,v,w,z\}$ for a vertex $z$ with list of size three, and either $uv$ is not crossed and
$uz,vz,wz\in E(G)$, or $uv$ is crossed by an edge incident with $z$, $zw\in E(G)$ and at least one of
$uz$ and $vz$ is an edge.
\item $V(G_2)=\{u,v,w,z\}$ for a vertex $z$ with list of size four adjacent to $u$, $v$, $w$ and incident with
an edge crossing $uv$.
\end{itemize}
\end{lemma}
\begin{proof}
For a contradiction, let us consider a $2$-chord $uvw$ that does not satisfy the conclusion of the lemma such that ($\star$) $uv$ is non-crossed
if possible, and subject to that, $G_2$ is maximal.
Let us distinguish the cases depending on whether $uv$ is crossed or not.

\textbf{First, suppose that $uv$ is not crossed.}  By the minimality of $B$, there exists an $L$-coloring $\vf$ of $G_1$.  Let $L_2(z)=\{\vf(z)\}$ for $z\in \{u,v,w\}$ and $L_2(z)=L(z)$
for $z\in V(G_2)\setminus \{u,v,w\}$, and let $B_2=(G_2,uvw,N\cap V(G_2), M\cap E(G_2), L_2)$.  Note that $G_2$ is not $L_2$-colorable, and thus the target $B_2$
is not valid.  This is only possible if it violates (P), (T), or (O).
If $B_2$ violates (P) or (T), then by Lemmas~\ref{lemma-cl-noseptr} and \ref{lemma-cl-nochord} the conclusion of Lemma~\ref{lemma-cl-2chord} holds for $uvw$.  Therefore, we conclude that
$B_2$ violates (O).

Since the obstruction in $B_2$ violating (O) contains a special subgraph with a vertex distinct from $v$ and $v\notin V(F)$, the condition (D) for $B$ implies
that $v\notin N$, and hence $|L(v)|=5$. By Lemmas~\ref{lemma-cl-noseptr} and \ref{lemma-cl-nochord-nocr} we also conclude that $G_2$ is equal to the obstruction.
Let $S$ be the set of $L$-colorings of $uvw$ that do not extend to an $L$-coloring of $G_2$. The inspection of the obstructions with $\ell(P)=2$ in Figure \ref{fig-badlists} shows that one of the following holds:
\begin{itemize}
\item[(R1)] there exists a set $A$ of at most two colors and $S$ contains only colorings $\psi$ such that $\psi(v)\in A$,
and furthermore, if $|A|=2$ then $|L(u)|\neq 3$ and $|L(w)|\neq 3$; or,
\item[(R2)] $S$ contains only colorings $\psi$ such that $\psi(u)=\psi(w)$, and neither $u$ nor $w$ has list of size three.
\end{itemize}
Indeed, by \refclaim{cl-colprop}, all obstructions except for \ob{M1} and \ob{C1} satisfy (R1) with $|A|=1$.
If $G_2$ is \ob{M1} or \ob{C1}, then neither $u$ nor $w$ has list of size three, by (M) together with the distance condition
and by (C).  The inspection of the colorings shows that if $G_2$ is \ob{C1}, then (R1) holds with $|A|=2$,
and if $G_2$ is \ob{M1}, then either (R1) holds with $|A|=2$, or (R2) holds (the latter is the case
when the two lists of size 3 are equal, i.e., $a=c$ in Figure~\ref{fig-badlists}).

If (R1) holds, then let $G'=G_1$, with the list assignment $L'$ such that $L'(v)=L(v)\setminus A$ and $L'(z)=L(z)$ for $z\ne v$.  Note that if
$|A|=2$, then $v$ has no neighbor in $G_1$ with list of size three by (R1) and by ($\star$)---the edge from $v$ to such a neighbor cannot
be crossed by the condition (D) for $B$ and the existence of the obstruction in $B_2$.  If (R2) holds, then let $G'=G_1+uw$ with the list
assignment $L'=L$.  Let $B'=(G',P,N\cap V(G'), M\cap E(G)', L')$.
In either case, since $G$ is not $L$-colorable, we conclude that $G'$ is not $L'$-colorable. The target $B'$ satisfies (D)---in the latter case, any path $Q$
between special subgraphs $H_1$ and $H_2$ using the added edge $uw$ gives rise to paths from $H_1$ and $H_2$ to the special subgraph $H$ of $G_2$,
and since $d(u,H)\le 2$ and $d(w,H)\le 2$, we have $\ell(Q)\ge 14+r(H_1)+r(H_2)+2r(H)-3>r(H_1)+r(H_2)+7$.  Furthermore, $B'$ satisfies (T) by Lemma~\ref{lemma-cl-no3inp},
and if $B'$ violated (C) or (O), then $v$ or
$uw$ would have to belong to a crossing or to an obstruction in $B'$, and the distance between its special subgraph and the special subgraph of $B_2$ would
be at most $4$, which contradicts the condition (D) for $B$.  Note that $B'$ cannot violate (P), as otherwise $u,w\in V(P)$ and $G_2$ is \ob{M1}, and
by Lemma~\ref{lemma-cl-basic} and \ref{lemma-cl-no3inp}, $v$ would have degree four and list of size five.
Therefore, $B'$ is a counterexample contradicting the minimality of $B$.

Hence, we conclude that Lemma~\ref{lemma-cl-2chord} holds for all $2$-chords without crossed edges.

\textbf{Suppose now that $uv$ is crossed} by an edge $xy$, where $x\in V(G_1)$ and $y\in V(G_2)$.  If $y=w$, then the conclusion of Lemma~\ref{lemma-cl-2chord} holds for $uvw$
by Lemmas~\ref{lemma-cl-crossnn} and \ref{lemma-cl-nosepnearc}, hence assume that $y\neq w$.  Furthermore, $x\neq w$ by Lemma~\ref{lemma-cl-crossnn}, and $uw\not\in E(G)$ by Lemma~\ref{lemma-cl-noseptr}.
Let $G_1'$ be the graph obtained from $G_1$ by adding the edges $ux$ and $vx$ (if they are not present already).
Note that this can be done without introducing any new crossings. Let $B'_1=(G'_1,P,N\cap V(G_1), M\cap E(G_1),L)$.
Since $u$, $v$ and $x$ are incident with a crossing in $G$, the target $B_1'$ satisfies (D) and (M). Furthermore, $B'_1$ does not contain any obstruction,
as its special subgraph would be at distance at most $2$ from the crossing. 
By Lemma~\ref{lemma-cl-nopcr}, $u$, $v$, and $x$ cannot belong to $P$, hence $B'_1$ satisfies (P) and (T).  We conclude that $B'_1$ is a valid target, and by the minimality of $B$,
there exists an $L$-coloring $\vf$ of $G'_1$.

Let $G'_2$ be the graph obtained from $G_2-uv$ by adding the vertex $x$ and edges $ux,vx,yx$. Consider the list assignment $L'_2$ for $G_2'$ such that
$L_2'(z)=\{\vf(z)\}$ for $z\in\{u,v,w,x\}$ and $L_2'(z)=L(z)$ otherwise.  Let $B'_2=(G'_2,uxvw,N\cap V(G_2'), M\cap E(G_2'),L'_2)$.  Note that $G'_2$ is not $L'_2$-colorable and that 
$B'_2$ satisfies (D).

Since $y\neq w$ and since $uw\notin E(G_2')$, the target $B'_2$ satisfies (P).  If $B'_2$ violates (T), then by Lemma~\ref{lemma-cl-noseptr} we have that $|L(y)|=3$ and $y$ is adjacent
to at least two of $u$, $v$ and $w$. In particular, $y\in V(F)$. Observe that if $vy\in E(G)$, then $wy\in E(G)$, since Lemma~\ref{lemma-cl-2chord} holds for the non-crossed $2$-chord
$yvw$.  Thus $wy\in E(G)$ in any case.  By Lemma~\ref{lemma-cl-nosepnearc} applied to the path $xywvu$ in $G$, and by Lemma~\ref{lemma-cl-crossnn}, we have $V(G_2)=\{u,v,w,y\}$ and the conclusion of Lemma~\ref{lemma-cl-2chord} holds for $uvw$.
Hence, we can assume that $B'_2$ satisfies (T).  By the minimality of $B$, the target $B'_2$ is not valid, and we conclude that it violates (O).
By Lemmas~\ref{lemma-cl-basic} and \ref{lemma-cl-noseptr}, $G'_2$ is equal to one of \ob{P1}, \ldots, \ob{P6}, but not to \ob{P3} since $x$ has degree $3$ in $G_2'$.

If $G'_2$ is \ob{P1}, then the conclusion of Lemma~\ref{lemma-cl-2chord} holds for $uvw$.  Otherwise, let us define
$S$ as the set of colorings $\psi$ of the path $uxvw$ that do not extend to an $L$-coloring of $G_2'$ and satisfy $\psi(u)\ne \psi(v)$.
The inspection of the obstructions and their problematic list assignments displayed in Figure \ref{fig-badlists} shows that either (R1)
or one of the following holds:
\begin{itemize}
\item[(R3)] $G'_2$ is \ob{P2} and there exists a color $c$ such that $S$ contains only colorings $\psi$ such that
either $\psi(u)=c$ and $\psi(x)=\psi(w)$, or $\psi(x)=c$ and $\psi(u)=\psi(w)$. Moreover, $|L(u)|\neq 3$ and $|L(w)|\neq 3$.
\item[(R4)] $G'_2$ is \ob{P4} and there exists a color $c$ such that $S$ contains only colorings $\psi$ satisfying either $\psi(v)=c$ or $\psi(x)=c$.  Moreover, $|L(u)|\ne3$.
\end{itemize}
Let us remark that for \ob{P2} we have (R1) if the colors $a,b,c,d$ in Figure \ref{fig-badlists} are different; we have (R3) if $b=d$ or $a=d$.
To argue for \ob{P4}, \ob{P5}, \ob{P6} we observe that $\psi(x)$ and $\psi(v)$ should be taken from the difference of the lists of the two neighbors of $u$
(so these are colors $b,c$ in Figure~\ref{fig-badlists}). This yields (R1) with the only exception in the case of \ob{P4}, where we cannot argue about $|L(w)|\ne3$,
so we need (R4) in this (and only this) case.

The conclusions that the specified vertices do not have lists of size three follow in all applicable cases by noting that otherwise either (C) or the distance
condition would be violated.  For example, for (R3) we argue as follows. Since $x$ has degree $3$ in $G_2'$,
the vertex $z$ of \ob{P2} with list of size 3 is not the vertex $y$, and $v,w$ are both adjacent to $z$. Since $|L(z)|=3$ and the edge $wz$ is
close to a crossing in $G$, we conclude that $wz\notin M$ and hence $|L(w)|\ne 3$. Since $|L(y)|=4$, (C) implies that $|L(u)|\ne 3$.

If (R1) holds, then let $G'_1=G_1$, let $M_1=M\cap E(G_1)$, let $N_1=N\cap V(G_1)$, and let $L_1$ be the list assignment obtained from $L$ by removing $A$ from the list of $v$.
If (R3) holds, then we let $G'_1=G_1+uw$ with the list assignment $L_1$ obtained from $L$ by removing $c$ from the list of $u$ (note that $|L(u)|\neq 1$
by Lemma~\ref{lemma-cl-nopcr}), let $N_1=N\cap V(G_1)$, and let $M_1=(M\cap E(G_1))\cup \{uz\}$ if $u$ has a neighbor $z$ with list of size three
and $M_1=M\cap E(G_1)$ otherwise.

If (R4) holds and $|L(x)|=5$, then let $G'_1=G_1$ with
the list assignment $L_1$ obtained by removing $c$ from the lists of $x$ and $v$, let $N_1=(N\cap V(G_1))\cup\{x\}$ and let $M_1=M\cap E(G_1)$. In all the cases, $B'_1=(G'_1,P,N_1,M_1,L_1)$
is a valid target.  Indeed, all the conditions except for (P), (T), and (O) are easy to argue.
(P) holds since $u\not\in V(P)$ by Lemma~\ref{lemma-cl-nopcr}.
Similarly, (T) follows by Lemmas~\ref{lemma-cl-nochord} and \ref{lemma-cl-no3inp}.  Finally, (O) holds since by the distance condition for $B$,
we could only create \ob{M1} or \ob{M2} if (R3) holds, and
\ob{N1}, \ob{N2}, or \ob{N3} if (R4) holds, and each of them is excluded by Lemma~\ref{lemma-cl-nochord} or \ref{lemma-cl-no3inp}, or by the condition $|L(u)| \neq 3$.
However, any $L_1$-coloring of $G'_1$ would extend to an $L$-coloring of $G$, and thus $B'_1$ is a counterexample contradicting the minimality of $B$.

Finally, consider the case that (R4) holds and $|L(x)|\in \{3,4\}$.  By Lemma~\ref{lemma-cl-crossnn}, all neighbors of $u$ distinct from $x$ belong to $G_2$.
By Lemma~\ref{lemma-cl-nopcr}, we have $u\not\in V(P)$, and since $\deg(u)\ge |L(u)|\ge 4$, we conclude that $u$ is adjacent to $x$ and $|L(u)|=4$.  Since $G'_2$ is \ob{P4},
every $L$-coloring of $x$, $v$ and $w$ extends to an $L$-coloring of the graph consisting of $G_2$, $x$, and all the edges between $x$ and $V(G_2)$.  We conclude that $G_1$ is not $L$-colorable, and thus $(G_1,P,N\cap V(G_1),M\cap E(G_1),L)$ contradicts the minimality of $B$.
\end{proof}

Similarly, one can prove the following.

\begin{lemma}\label{lemma-cl-c2chord}
Let $B=(G,P,N,M,L)$ be a minimal counterexample, and let $F$ denote the outer face of $G$.
Let $u,w\in V(F)$ be distinct vertices which are not crossing-adjacent and neither of which is an internal vertex of $P$. Suppose that $v\not\in V(F)$ is a vertex adjacent to $w$
and crossing-adjacent to $u$. Let $c$ be the
closed curve consisting of $vw$, parts of the crossed edges incident with $u$ and $v$,
and a part of the boundary of $F$ between $u$ and $w$ that does not contain $P$, and let $G_2$ be the subgraph of $G$ consisting of vertices and edges fully drawn in the
closed disc bounded by $c$.  Then $G_2$ does not contain the crossing and satisfies one of the following:
\begin{itemize}
\item[\rm (a)] $V(G_2)=\{u,v,w\}$ and $uw\in E(G)$, or
\item[\rm (b)] $V(G_2)=\{u,v,w,z\}$, $|L(z)|=3$ and $z$ is adjacent to $u$, $v$ and $w$.
\end{itemize}
\end{lemma}
\begin{proof}
By Lemma~\ref{lemma-cl-2chord}, it suffices to consider the case that $uv\not\in E(G)$.
Let $c_1$ be the closed curve consisting of $vw$, parts of the crossed edges incident with $u$ and $v$,
and a part of the boundary of $F$ between $u$ and $w$ that contains $P$.
Let $G_1$ be the subgraph of $G$ drawn in the closed disk bounded by $c_1$.
Let $G'_1$ be the graph obtained from $G_1$ as follows:  If $uw\in E(G)$, then we add the edge $uw$.
If $u$, $v$ and $w$ have a common neighbor $z$ with list of size three, then we add $z$ and incident edges.
If $V(G'_1)=V(G)$, then (a) or (b) holds.

Otherwise, there exists an $L$-coloring $\vf$ of $G'_1$ by the minimality of $B$.  Let $G'_2=G_2+uv$ with the list assignment $L'$ such that $L'(v)=\{\vf(v)\}$,
$L'(w)=\{\vf(w)\}$, $L'(u)=\{a\}$ for a new color $a$, $L'(x)=(L(x)\setminus \{\vf(u)\})\cup\{a\}$ for
each neighbor $x$ of $u$ distinct from $v$ and $w$ and $L'(x)=L(x)$ for all other vertices $x$ of $G'_2$.
Note that $G'_2$ is not $L'$-colorable, and by the minimality of $G$, it follows that $B'=(G'_2,uvw,N\cap V(G'_2), M\cap E(G'_2), L')$
is not a valid target.  Note that $B'$ satisfies (P) and (T) by the construction of $G_1'$ and the choice of $\vf$.
The only other condition that can be violated by $B'$ is (O).
By the distance condition, the only obstruction that can appear in $B'$ is \ob{C1}.  However, letting $t$ be the neighbor of $u$ in $G_2$ with list of size four,
the $2$-chord $wvt$ contradicts Lemma~\ref{lemma-cl-2chord}.
\end{proof}

Let us now introduce a way of defining list assignments that will be used throughout the rest of the paper.
Let $(G,P,N,M,L)$ be a valid target.
Let $\vf$ be any proper partial $L$-coloring of $G$ such that $\vf(v)\not\in L(p)$ for every pair of adjacent vertices
$v\in\dom(\vf)$ and $p\in V(P)$. 
For each vertex $z\in V(G)\setminus V(P)$, let $$R_z=\bigcup_{p\in V(P)\setminus \dom(\vf), zp\in E(G)} L(p).$$
For $z\in V(P)$, let $R_z=\emptyset$.
We define $\lf$ to be the list assignment such that
$$\lf(z)=\Bigl(L(z)\setminus \{\vf(x):x\in\dom(\vf),xz\in E(G)\}\Bigr)\cup R_z$$
for each $z\in V(G)$.  Let us also define $\gf=G-\dom(\vf)$.
That is, from the list of each vertex $z$ we remove the colors of its neighbors according to $\vf$, except for those
colors that are also forbidden by the neighbors of $z$ in $P\cap \gf$.

Consider any $\lf$-coloring $\psi$ of $\gf$.  We claim that the combination of $\vf$ with $\psi$ is a proper
$L$-coloring of $G$.  Indeed, for any $z\in V(\gf)$, we clearly have $\psi(z)\not\in R_z$, and thus $\psi(z)\in \lf(z)$
is different from the colors of the neighbors of $z$ in $\dom(\vf)$.  Since $G$ is not $L$-colorable,
we conclude that $\gf$ is not $\lf$-colorable.

Suppose now that $(G,P,N,M,L)$ is a valid target and that $G$ contains a subgraph $H$ isomorphic to one of the graphs drawn in Figure~\ref{fig-obst} such that the subgraph of $H$
corresponding to full-circle vertices is equal to $P$, triangle vertices have lists of size {\em at least} three, square vertices have lists
of size {\em at least} four and pentagonal vertices have lists of size five.  Then we say that $H$ is a {\em near-obstruction}.

\begin{lemma}\label{lemma-cl-nearobst}
Let $B=(G,P,N,M,L)$ be a minimal counterexample, and let $F$ denote the outer face of $G$.  If $H$ is a near-obstruction in $B$,
then $H$ is isomorphic to one of\/ \ob{M1}, \ob{N2}, \ob{N3} or \ob{P3}.  Furthermore, $|(V(H)\cap V(F))\setminus V(P)|\le 1$, and if
$(V(H)\cap V(F))\setminus V(P)\ne \emptyset$, then $H$ is \ob{N2} or \ob{N3}.
\end{lemma}
\begin{proof}
By Lemma~\ref{lemma-cl-no3inp}, $H$ is isomorphic to one of \ob{M1}, \ob{N2}, \ob{N3}, \ob{C2}, \ob{C3}, \ob{C4}, \ob{C5} or \ob{P3}.
Let $k=\ell(P)$ and $P=p_0p_1\ldots p_k$.  Note that $G\neq H$, since otherwise $G$ is $L$-colorable by (O).

Let us exclude several of the possible near-obstructions.
\begin{itemize}
\item If $H$ is \ob{C5}, then let $w$ be the vertex of the outer face of $H$ not belonging to $P$.  By Lemma~\ref{lemma-cl-2chord},
$G$ is obtained from $H$ either by adding the edge $p_0p_2$, or a vertex $z$ of degree three adjacent to $p_0$, $w$, and $p_2$.
However, the distance condition implies that $w\notin N$, so that $|L(w)|=5$. In both cases, this implies that $G$ is
$L$-colorable, which is a contradiction.

\item If $H$ is \ob{C2}, then let $p_0w_1w_2p_2$ be the path in the outer face of $H$.  If $w_1,w_2\in V(F)$, then $V(G)=V(H)$ and $G$ is $L$-colorable by
(O). Hence, by symmetry we can assume that $w_2\not\in V(F)$, and thus $|L(w_2)|=5$.  If $w_1\in V(F)$, then since $w_2$ has degree at least $5$, by Lemma~\ref{lemma-cl-2chord} we
have that $G$ is obtained from $H$ by adding a vertex $z$ is adjacent to $w_1$, $w_2$ and $p_2$. However, then
$G$ is $L$-colorable.  Therefore, $w_1\not\in V(F)$.

Let $\vf$ be an $L$-coloring of $H$, let $G_2$ be the $p_0w_1w_2p_2$-component of $G$ that does not contain $P$, with the list assignment $L_2$
obtained from $L$ by setting $L_2(x)=\{\vf(x)\}$ for $x\in\{p_0,w_1,w_2,p_2\}$, and let $B_2=(G_2,p_0w_1w_2p_2,N\cap V(G_2),M\cap E(G_2),L)$.
Since $\vf$ does not extend to an $L$-coloring of $G_2$, it follows that $B_2$ is not a valid target, which is only possible
if it violates (P), (T) or (O).  Since both $w_1$ and $w_2$ have degree at least $5$ in $G$, Lemma~\ref{lemma-cl-noseptr} implies that $p_0w_2\notin E(G)$ and
$w_1p_2\notin E(G)$, and thus $B_2$ satisfies (P).

Suppose that $B_2$ violates (T). Then a vertex $z$ with $|L(z)|=3$ is adjacent to three vertices among
$p_0$, $w_1$, $w_2$ and $p_2$.
If it is adjacent to all four of them, then $B$ contains \ob{C5} as a near-obstruction, which has already been excluded. Otherwise, since $w_1$ and $w_2$ have degree at
least 5 in $G$, Lemma~\ref{lemma-cl-nosepq} implies that $z$ is not adjacent to $p_0, w_1, p_2$. By symmetry, we may assume that $z$ is adjacent to $p_0,w_1,w_2$.
Then Lemma~\ref{lemma-cl-2chord} applied to the 2-chord $zw_2p_2$ shows that there is a vertex $z'$ adjacent to $z$ with $|L(z')|=3$, and
thus the edge $zz'$ contradicts either (M) or (D) in $B$.  Hence, $B_2$ satisfies (T).

Finally, if $B_2$ violates (O), then the obstruction is equal to one of \ob{P1}, \ob{P2}, \ob{P3},
\ob{P4}, \ob{P5} or \ob{P6}.  However, then it is easy to see (by comparing bad lists for the obstructions) that $G$ is $L$-colorable,
which is a contradiction.

\item If $H$ is \ob{C3}, then let $w_1$ be the vertex of $H$ drawn by the triangle and $w_2$ the vertex of $P$ that is not adjacent to it in $G$.  If $H$ is
\ob{C4}, then let $w_1$ and $w_2$ be the vertices of $H$ drawn by triangles.  By symmetry, we can assume that $w_1$ is the neighbor of $p_2$.  Let
$w_1x_1x_2w_2$ be the path in $H$ formed by neighbors of $p_1$.  Note that $|L(w_i)|\in\{1,5\}$ by Lemma~\ref{lemma-cl-nochord}.

Choose an $L$-coloring
$\vf$ of the subgraph of $G$ induced by $V(P)\cup \{w_1,w_2\}$ such that $\vf(w_1)\neq \vf(w_2)$ and either $|\lf(x_1)|\ge 4$ or
$\lf(x_1)\neq\lf(x_2)$. Note that this is possible since $|L(w_1)|=5$.  Let $G'=G-\{p_1,x_1,x_2\}+w_1w_2$ with the list assignment $L'$ such that
$L'(z)=\{\vf(z)\}$ for $z\in \{w_1,w_2\}$ and $L'(z)=L(z)$ otherwise.  Observe that every $L'$-coloring of $G'$ extends to an $L$-coloring of $G$, and thus
$G'$ is not $L'$-colorable.  Let $P'=w_2w_1p_2$ if $H$ is \ob{C3} and $P'=p_0w_2w_1p_2$ if $H$ is \ob{C4}, and let $B'=(G',P',N\cap V(G'), M\cap E(G'), L')$.
By the minimality of $B$, we conclude that $B'$ is not a valid target.

Note that the choice of $\vf$ ensures that $B'$ satisfies (P).  Observe that $B'$ may only violate (T) or (O).
In the former case, by symmetry we can assume that there exists a vertex $z\in V(G)$ such that $|L(z)|=3$ and $z$ is adjacent to
$p_2$, $w_1$ and either $w_2$ or $p_0$.  It follows that $G$ contains a separating $4$-cycle formed by non-crossed edges, and by Lemma~\ref{lemma-cl-basic}(d)
the interior of this $4$-cycle contains $K_4$.  By Lemmas~\ref{lemma-cl-nochord} and \ref{lemma-cl-2chord}, there are no other vertices in $G$. Now, it is easy
to see that the resulting graph $G$ is $L$-colorable.  This is a contradiction.

Therefore, $B'$ violates (O).  By the condition (D) for $G$, we conclude that $\ell(P')=3$ (and thus $H$ is \ob{C4}) and $B'$ contains one of \ob{P1}--\ob{P6}.
Note that the edge $w_1w_2$ is contained in a triangle in $G'$; let $z$ be the
common neighbor of $w_1$ and $w_2$.  By Lemma~\ref{lemma-cl-basic}, the $4$-cycle $w_1zw_2p_1$ surrounds $K_4$ in $G$.  However, then $G$ is obtained
from one of the obstructions \ob{P1}--\ob{P6} with the precolored path $p_0w_2w_1p_2$ by
adding the vertex $p_1$ joined to the vertices $p_0,w_1,w_2,p_2$, deleting the edge $w_1w_2$, and adding $K_4$ inside the 4-cycle $w_1zw_2p_1$, and all such graphs
are easily seen to be $L$-colorable.  This is a contradiction.
\end{itemize}

We conclude that $H$ is one of \ob{M1}, \ob{N2}, \ob{N3}, and \ob{P3}. If $H$ is \ob{M1} or \ob{P3}, then none of the vertices in $V(H)\setminus V(P)$
belongs to $F$ since this would contradict Lemma~\ref{lemma-cl-nochord}. In the other cases, at most one of the vertices of $V(H)\setminus V(P)$ can belong to $F$ by the same
reason.
\end{proof}

Next, we analyze the part of the boundary of the outer face of a minimal counterexample next to the precolored path.

\begin{lemma}\label{lemma-nearby}
Let $B=(G,P,N,M,L)$ be a minimal counterexample, and let $F$ denote the outer face of $G$.
Let $k=\ell(P)$ and let $p_k\ldots p_1p_0v_1v_2\ldots v_s$ be the vertices contained in the boundary of $F$ in the cyclic order
around it.  Then $k\ge 2$, $s\ge 3$ and $v_1v_2,v_{s-1}v_s\in E(G)$.  Furthermore, if $|L(v_1)|>3$ and $|L(v_2)|>3$, then $v_1$ and $v_2$
are crossing-adjacent and $|L(v_1)|=|L(v_2)|=4$.
\end{lemma}
\begin{proof}
Let us define $v_{s+1}=p_k$, $v_{s+2}=p_{k-1}$, \ldots.
Note that for $i=1,\ldots, s$, we either have $v_iv_{i+1}\in E(G)$, or $v_i$ and $v_{i+1}$ are crossing-adjacent.
Lemma~\ref{lemma-cl-nopcr} implies that $p_0v_1,p_kv_s\in E(G)$.

If $k<2$, then let $P'=v_1p_0\ldots p_k$ and let $L'$ be the list
assignment obtained from $L$ by setting $L'(v_1)=\{a\}$ for a color $a\in L(v_1)$ chosen so that $B'=(G,P',N,M,L')$ satisfies
(P) and (O).  Then it is easy to see that $B'$ is a valid target, and by the minimality of $B$, we conclude that $G$ is $L'$-colorable.
This is a contradiction, since this also gives an $L$-coloring of $G$.  Therefore, we have $k\ge 2$.

If $s=0$, then let $\vf$ be the $L$-coloring of $p_0$; observe that $(\gf,p_1\ldots p_k,N\cap V(\gf),\emptyset,\lf)$ is
a counterexample contradicting the minimality of $B$.  Hence, $s\ge 1$.

Suppose for a contradiction that $s=1$. Let $\vf$ be a partial coloring that assigns a color in
$L(v_1)\setminus (L(p_0)\cup L(p_k))$ to $v_1$. If $v_1$ is adjacent to a vertex $x$ by a crossed edge, then let $N'=(N\cap V(\gf))\cup\{x\}$,
otherwise let $N'=N\cap V(\gf)$.  By the minimality of $B$, we conclude that the target $B'=(\gf,P,N',\emptyset,\lf)$
violates (O).  Let $H$ be an obstruction in $B'$; by Lemma~\ref{lemma-cl-nearobst}, $H$ is one of \ob{M1}, \ob{N2}, \ob{N3} or \ob{P3}.
Note that $H$ contains a vertex $z$ such that $|\lf(z)|=3$, and by the choice of $\vf$, we conclude that $z\in N$ and $zv_1\in E(G)$.
However, since $H$ also contains a special subgraph, we obtain a contradiction with the condition (D) for $B$.
Therefore, $s\ge 2$.

Suppose for a contradiction that $v_1v_2\not\in E(G)$, and thus $v_1$ and $v_2$ are crossing-adjacent.
Note that $\ell(P)=2$ by the distance condition.
Let $\vf$ be a partial coloring that assigns a color from $L(v_1)\setminus L(p_0)$ to $v_1$
and the color from $L(p_0)$ to $p_0$.
Let $y$ be the vertex adjacent to $v_1$ by the crossed edge, and let $B'=(\gf,p_1p_2,(N\cap V(\gf))\cup \{y\}, M\cap E(\gf),\lf)$.
By the minimality of $B$, we conclude that $B'$ is not a valid target.  By Lemma~\ref{lemma-cl-nochord}, we have $|L(y)|=5$,
and if $|\lf(y)|<4$, then $y$ is adjacent to $p_0$; however, by Lemma~\ref{lemma-cl-noseptr},
$v_2$ would be adjacent to $p_0$, contrary to Lemma~\ref{lemma-cl-nochord}.  Therefore, we have $|\lf(y)|\ge 4$.
By Lemmas~\ref{lemma-cl-nochord} and \ref{lemma-cl-no3inp}, we conclude that only the condition (M) can be violated by $B'$.
In that case, $p_0$ and $v_1$ have a common neighbor $u\neq y$ adjacent to a vertex $w$ with $|L(w)|=3$.  This contradicts
Lemma~\ref{lemma-cl-2chord}.  Therefore, $v_1v_2\in E(G)$, and by symmetry, $v_{s-1}v_s\in E(G)$.

Suppose for a contradiction that $s=2$.  By symmetry, assume that if $v_2$ is incident with a crossed edge,
then $v_1$ is incident with a crossed edge as well.
If $v_1v_2\in M$, then let $\vf$ be an $L$-coloring of $v_1$ and $v_2$ such that $\vf(v_1)\not\in L(p_0)$ and
$\vf(v_2)\not\in L(p_k)$.  Otherwise, let $\vf$ be a coloring of $v_1$ by a color in $L(v_1)\setminus L(p_0)$ such that if
$|L(v_2)|=3$, then $\vf(v_1)\not\in L(v_2)\setminus L(p_k)$. Note that this is possible by Lemma~\ref{lemma-cl-basic}(f).
Let us remark that when $|L(v_2)\setminus \{\vf(v_1)\}|=2$, then $L(p_k)=\{\vf(v_1)\}$ and $\lf(v_2)=L(v_2)$
by the definition of $\lf$, and thus we always have $|\lf(v_2)|\ge 3$.  If $v_1$ is incident with a crossed
edge $v_1x$, then let $N'=N\cup \{x\}$; if $v_1$ is adjacent to a vertex $y\in N$, then let $N'=N\setminus\{y\}$; otherwise, let
$N'=N\cap V(\gf)$.
If $v_1$ and $v_2$ have a common neighbor $z$ belonging to $N$ and $|\lf(v_2)|=3$, then let $M'=M\cup \{v_2z\}$;
otherwise let $M'=M\cap E(\gf)$.  Let $B'=(\gf,P,N',M',\lf)$; by the minimality of $B$, the target $B'$ is not valid.

The choice of $\vf$, $M'$ and $N'$ ensures that $B'$ satisfies (S), (N), (M), (P), and (D).  It satisfies (T) by Lemma~\ref{lemma-cl-no3inp}.
If $B'$ violated (C), then $v_2$ would have to be incident with a crossing, and by the choice of $v_1$ and the distance condition,
the vertex $v_1$ would be incident with the same crossing, which then would not appear in $\gf$.  Therefore, $B'$ satisfies (C).
Hence, $B'$ violates (O) and $G$ contains a near-obstruction $H$.  By Lemma~\ref{lemma-cl-nearobst}, $H$ is \ob{M1}, \ob{N2}, \ob{N3}
or \ob{P3}.  Observe that $v_1v_2\not\in M$, since otherwise the distance between $v_1v_2$ and the special subgraph of $H$ (which
is also special in $G$) is at most $3$.  Every vertex with list of size three according to $\lf$ either belongs to $N$ or is
equal to $v_2$.  If $v_2\not\in V(H)$, then $H$ contains only one vertex with list of size three, hence $H$ is \ob{N2}.  However,
then $N$ contains two adjacent vertices, which is a contradiction.  Similarly, we exclude the case that $v_2\in V(H)$ and $H$ is
\ob{N3} or \ob{P3}.  Therefore $v_2\in V(H)$ and $H$ is \ob{M1} or \ob{N2}.  The former is excluded by Lemma~\ref{lemma-cl-nochord}.
If $H$ is \ob{N2}, then we have $V(G)=V(H)\cup \{v_1\}$ by Lemma~\ref{lemma-cl-noseptr}.  If $v_1$ is
incident with a crossed edge, then $G$ contains \ob{C2}.  On the other hand, if $v_1$ is not incident with a
crossed edge, then $|L(v_1)|=3$, $|L(v_2)|=4$, $|N|=1$ and $G$ is $L$-colorable.  This is a contradiction, and thus $s\ge 3$.

Suppose for a contradiction that $v_1$ and $v_2$ are not crossing-adjacent, $|L(v_1)|>3$, and $|L(v_2)|>3$.
We remove a color from the list of $v_1$; let $B'$ denote the resulting target.
If some edge $v_1x$ crosses an edge $e$, then $|L(x)|=5$ by Lemma~\ref{lemma-cl-nochord}, and both vertices incident
with $e$ have list of size five by Lemmas~\ref{lemma-cl-nopcr} and \ref{lemma-cl-crossnn}, hence $B'$ satisfies (C).
By Lemma~\ref{lemma-cl-nearobst}, no obstruction arises (since all vertices with lists of size three or four in the new list assignment
are contained in $V(F)\cup N$, implying that the corresponding near-obstruction would have at least two vertices in $V(F)\setminus V(P)$).
We conclude that $B'$ contradicts the minimality of $B$; hence, if $|L(v_1)|>3$ and $|L(v_2)|>3$, then $v_1$ and $v_2$
are crossing-adjacent.  A similar argument shows that in this case $|L(v_1)|=|L(v_2)|=4$.
\end{proof}

Consider a minimal counterexample $B=(G,P,N,M,L)$ with outer face $F$.
Let $k=\ell(P)$ and let $p_k\ldots p_1p_0v_1v_2\ldots v_s$ be the vertices contained in the boundary of $F$ in the cyclic order
around it, with $v_{s+1}=p_k$, $v_{s+2}=p_{k-1}$, \ldots.  Suppose that $v_1v_2,v_2v_3\not\in M$.
If $|L(v_1)|=3$ or $|L(v_2)|=3$, then let the set $X\subseteq V(F)\setminus V(P)$ and its
partial $L$-coloring $\vf$ be defined as in (X1)--(X4) in the proof of Theorem~\ref{thm-basic}.  Let us remark that the choice of $X$ and $\vf$
is performed independently on whether $v_2v_3$ or $v_3v_4$ is an edge, or whether they are crossing-adjacent
(note that when $v_3$ is crossing-adjacent to $v_2$ or $v_4$, then the case (X4) cannot apply by (C)).
If $|L(v_1)|>3$ and $|L(v_2)|>3$, then $v_1$ and $v_2$ are crossing-adjacent by Lemma~\ref{lemma-nearby}, and we define $X$ and $\vf$ as follows.
\begin{itemize}
\item[(X5)] If $|L(v_1)|=|L(v_2)|=4$ and $|L(v_3)|\neq 3$, then $X=\{v_1\}$ and $\vf(v_1)\in L(v_1)\setminus L(p_0)$ is chosen arbitrarily.
\item[(X6)] If $|L(v_1)|=|L(v_2)|=4$ and $|L(v_3)|=3$, then $X=\{v_2\}$ and $\vf(v_2)\in L(v_2)\setminus L(v_3)$ is chosen arbitrarily.
\end{itemize}
Let $m$ denote the largest index such that $v_m\in X$.
Let us note that $m=1$ in (X1) and (X5), $m=3$ in (X4), and $m=2$ otherwise.
Also, $X=\dom(\vf)$ in all cases except for (X4b), when $X=\{v_1,v_2,v_3\}$ and $\dom(\vf)=\{v_1,v_3\}$.
We say that the triple $(X,\vf,m)$ is the \emph{probe for $B$ at the $p_0$ side of $P$}.
The \emph{probe for $B$ at the $p_k$ side of $P$} is defined symmetrically, exchanging the role of $p_0$ with $p_k$, $v_1$ with $v_s$, \ldots

\begin{figure}
\begin{center}
\includegraphics[width=95mm]{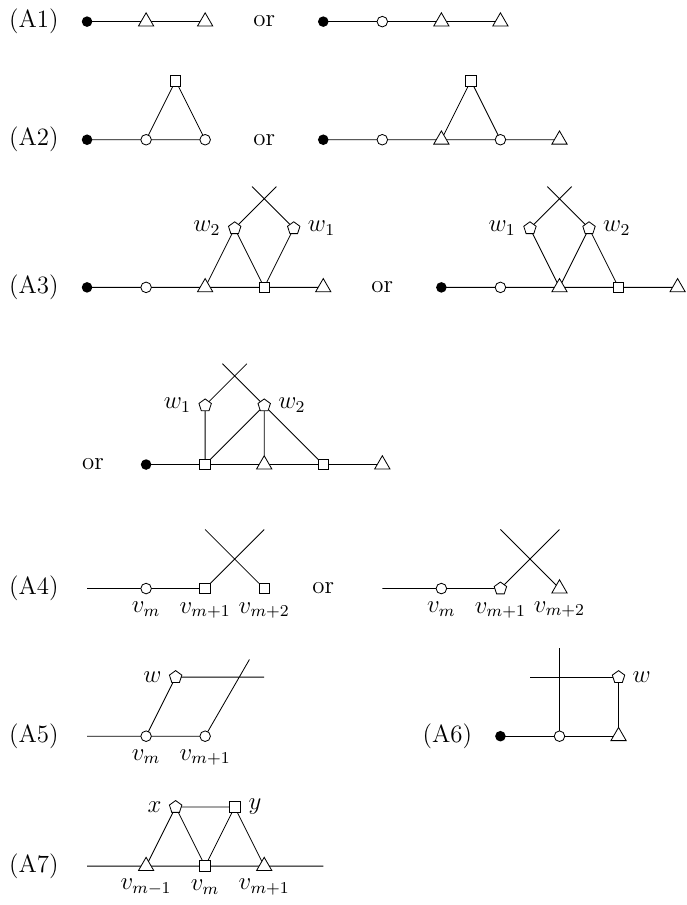}
\end{center}
\caption{Possible outcomes of Lemma~\ref{lemma-cl-a}.}
\label{fig-a}
\end{figure}

\begin{lemma}\label{lemma-cl-a}
Consider a minimal counterexample $B=(G,P,N,M,L)$ with outer face $F$.
Let $k=\ell(P)$ and let $p_k\ldots p_1p_0v_1v_2\ldots v_s$ be the vertices contained in the boundary of $F$ in the cyclic order
around it.  If $v_1v_2,v_2v_3\not\in M$, then let $(X,\theta,m)$ be the probe for $B$
at the $p_0$ side of $P$.  One of the following cases holds, see Figure~\ref{fig-a} for the illustration of the possibilities.
\begin{itemize}
\item[\rm (A1)] $v_1v_2\in M$ or\/ $v_2v_3\in M$.
\item[\rm (A2)] Either $v_1$ and $v_2$, or two distinct vertices in $\dom(\theta)$ have a common neighbor in $N$.
\item[\rm (A3)] There exists a crossing $q$ and two crossing-adjacent vertices $w_1,w_2\in V(G_q)$ such that
$V(G_q)\cap X=\emptyset$, $w_1$ has a neighbor in $\dom(\theta)$ and $w_2$ has two neighbors in $\dom(\theta)$.
\item[\rm (A4)] $v_mv_{m+1}\in E(G)$, there exists a crossing $q$ such that $V(G_q)\cap X=\emptyset$ and $v_{m+1},v_{m+2}\in V(G_q)$,
and either $|L(v_{m+1})|=|L(v_{m+2})|=4$ or $|L(v_{m+1})|=5$ and $|L(v_{m+2})|=3$.
\item[\rm (A5)] $v_mv_{m+1}\in E(G)$, $|L(v_{m+1})|\in\{3,4\}$ and there exists a crossing $q$ such that $V(G_q)\cap X=\emptyset$,
$v_{m+1}\in V(G_q)$ and a neighbor $w\not\in V(F)$ of $v_m$ is crossing-adjacent to $v_{m+1}$.
\item[\rm (A6)] $v_1\not\in X$ and there exists a crossing $q$ such that $V(G_q)\cap X=\emptyset$, $v_1\in V(G_q)$
and a neighbor $w\not\in V(F)$ of $v_2$ is crossing-adjacent to $v_1$.
\item[\rm (A7)] $|X|\ge 2$ and there exists a path $v_{m-1}xyv_{m+1}$, where $x$ and $y$ are neighbors of $v_m$ and $y\in N$.
\end{itemize}
\end{lemma}
\begin{proof}
We can assume that (A1) does not hold, and thus the probe $(X,\theta,m)$ is defined.

Let us consider the graph $G_0'=G-\dom(\theta)$ with the list assignment $L_\theta$, and let $G'$ be the graph obtained from $G'_0$ by repeatedly removing
vertices whose list is larger than their degree. If $\dom(\theta)\ne X$, then we have case (X4b) and $X\setminus \dom(\theta)=\{v_2\}$.
In this case, if $v_2$ were incident with a crossing, then Lemma~\ref{lemma-cl-noseptr} implies that an edge incident with $v_2$ crosses
an edge incident with $v_1$ or $v_3$, and since $|L(v_1)|=|L(v_3)|=4$ and $|L(v_2)|=3$, this would contradict (C).
Hence, $v_2$ is not incident with a crossing and its degree in $G_0'$ is 1, and since $|L_\theta(v_2)|\ge 2$, the vertex $v_2$
is not present in $G'$.  This shows that $G'\subseteq G-X$.  Observe also that $G'$ is not $L_\theta$-colorable.

Note that every vertex $v\in V(G')$ not incident with the outer face of $G'$ satisfies $|L_\theta(v)|\ge 4$ (and $|L_\theta(v)|=4$ only if
either $v\in N$ or $v$ is incident with a crossed edge with the other end in $\dom(\theta)$).  Let $N'$ be the set of vertices
$v\in V(G')$ not incident with the outer face of $G'$ such that $|L_\theta(v)|=4$.

If a vertex $v\in V(G')\setminus V(P)$ satisfies $|L_\theta(v)|\le 2$, then
$v\not\in V(F)$ by the choice of $X$ and $\theta$ and by Lemma~\ref{lemma-cl-nochord}.  Since $|\dom(\theta)|\le 2$, it follows
that $v\in N$ and $v$ has two neighbors in $\dom(\theta)$, and thus (A2) holds.
Hence, we can assume that $|L_\theta(v)|\ge 3$ for all $v\in V(G')\setminus V(P)$.

Let $M'$ consist of all edges of $G'$ that join vertices with list of size three, and let $B'=(G',P,N',M',L_\theta)$.
Note that $B'$ satisfies (S), (N), and (M) by the choice of $N'$ and $M'$.  Since $L_\theta(p)=L(p)$ for all $p\in V(P)$,
the target $B'$ satisfies (P).  By Lemma~\ref{lemma-cl-no3inp}, $B'$ satisfies (T).

Now, let us consider property (C). Let $q$ be a crossing in $G'$ and suppose that (C) is violated at $q$,
i.e., there exist distinct $u,v\in V(G_q)$ such that $|L_\theta(u)|=3$ and $|L_\theta(v)|\in \{3,4\}$.
Note that each of $u$ and $v$ is either incident with $F$ or adjacent to a vertex in $\dom(\theta)$.
Let us distinguish several cases.
\begin{itemize}
\item If both $u$ and $v$ belong to $F$, then by Lemmas~\ref{lemma-cl-nochord} and \ref{lemma-cl-crossnn} we have that
$u$ and $v$ are crossing-adjacent,
$\{u,v\}=\{v_{m+1},v_{m+2}\}$ and $L(v_{m+2})=L_\theta(v_{m+2})$.  It follows that $|L(v_{m+1})|\neq 3$ and that (A4) holds.
\item If $|\{u,v\}\cap V(F)|=1$ and $u$ and $v$ are not crossing-adjacent, then since $V(G_q)\cap X=\emptyset$,
Lemma~\ref{lemma-cl-2chord} implies that (A4) holds.
\item If $u\in V(F)$, $v\not\in V(F)$, and $u$ and $v$ are crossing-adjacent,
then we apply Lemma~\ref{lemma-cl-c2chord}. The outcome (a) of Lemma~\ref{lemma-cl-c2chord} gives (A5) or (A6). The outcome (b) gives a vertex $w\in X$
that is adjacent to $v$, and a vertex $z$ with $|L(z)|=3$ that is adjacent to $u$, $v$ and $w$. Therefore, $|L(u)|\ne3$, so $u$ has a neighbor in $X$,
necessarily equal to $z$.  Note that $|X|\ge 2$ and $|L(z)|=3$ can only be satisfied
in the subcase (X4a) of the definition of $X$ with $z=v_2$, $w=v_3$ and $u=v_1$, and thus (A6) holds.
\item If $u\not\in V(F)$, $v\in V(F)$, and $u$ and $v$ are crossing-adjacent, then $u$ has two neighbors
in $\dom(\theta)$.  
By Lemma~\ref{lemma-cl-c2chord}, one of the neighbors of $u$ in $X$ is also adjacent to $v$ and has list of size three, and by the choice
of $X$, we conclude that (A6) holds.
\item Finally, if $u,v\not\in V(F)$, then $|\dom(\theta)|=2$, $u$ is adjacent to both vertices in $\dom(\theta)$ and $v$ is adjacent to
at least one of them.  Hence, $u$ and $v$ are crossing-adjacent by Lemma~\ref{lemma-cl-noseptr}
and the fact that $V(G_q)\cap X=\emptyset$, and (A3) holds.
\end{itemize}

Therefore, we can assume that $B'$ satisfies (C).  Let us now consider the newly created special subgraphs in $B'$.
\begin{itemize}
\item If $v\in N'\setminus N$, then $v$ is adjacent to a vertex of $\dom(\theta)$ by an edge containing a crossing $q$, and only one vertex of $G_q$ belongs
to $X$.  Therefore, $|N'\setminus N|\le 1$.
\item If $xy\in M'\setminus M$, then $|L_\theta(x)|=|L_\theta(y)|=3$, and at least one of $x$ and $y$ has a list of size at least $4$ in the list
assignment $L$.

Suppose that $x,y\not\in N$.  If $x,y\not\in V(F)$, then both $x$ and $y$ have two neighbors in $\dom(\theta)$.
By Lemma~\ref{lemma-cl-2chord}, this implies that $x$ and $y$ are crossing-adjacent in $G$ via the edges joining $x,y$ with $\dom(\theta)$.
If $x,y\in V(F)$, then by Lemma~\ref{lemma-cl-nochord} we can assume that $x=v_{m+1}$ and $y=v_{m+2}$; but then $|L_\theta(x)|\neq 3$ or
$|L_\theta(y)|\neq 3$ by the choice of $X$, which is a contradiction.  Finally, suppose that say $x\in V(F)$ and $y\not\in V(F)$; then
$y$ has two neighbors in $\dom(\theta)$ and, in particular, we have cases (X2) or (X4).
By Lemma~\ref{lemma-cl-2chord}, we have $x\in \{v_1,v_{m+1}\}$.  If $x=v_1$, then $y$ would be a
common neighbor of $v_1$, $v_2$ and $v_3$, contradicting the choice of $X$ (assumptions of (X4b) are satisfied,
hence we would have $v_1\in X$).  If $x=v_{m+1}$, then since $|L(v_m)|\ge 4$, Lemma~\ref{lemma-cl-2chord} implies that
$|L(v_m)|=4$ and $v_m$ is incident with an edge crossing either the edge $v_{m-1}y$ or the edge $v_{m+1}y$.
However, by the choice of $X$ we have $|L(v_{m-1})|=|L(v_{m+1})|=3$, contradicting (C).

Therefore, either $\{x,y\}\cap N\neq\emptyset$, or there exists a crossing $q$ such that $\{x,y\}=V(G_q)\setminus X$.
\end{itemize}

It follows that $d(S_1,S_2)\ge 7+r(S_1)+r(S_2)$ whenever $S_1$ is a special subgraph of $G$ that is also special in $G'$
and $S_2$ is any special subgraph of $G'$.  Suppose now that $S_1$ and $S_2$ are both distinct newly created special subgraphs in $G'$.
Note that $|N'\setminus N|\le 1$ and if $N'\setminus N\neq\emptyset$, then $M'\setminus M=\emptyset$.  It follows that
$S_1,S_2\in M'\setminus M$.  As proved in the previous paragraph, each edge in $M'\setminus M$ is incident with a special subgraph in $G$
that is adjacent to $X$. By the distance condition, we conclude that there exists a path $xyz$ in $G'$ such that $|L_\theta(x)|=|L_\theta(y)|=|L_\theta(z)|=3$
and $y\in N$.  Note that at most one of $x$ and $z$ can have two neighbors in $\dom(\theta)$, as otherwise
$G$ would contain a crossing at distance at most one from $y$; thus we may assume that $x\in V(F)$.  Since $y$ has a neighbor in
$\dom(\theta)$, Lemma~\ref{lemma-cl-2chord} implies that
$x\in \{v_1,v_{m+1}, v_{m+2}\}$.  If $x=v_{m+2}$, then we would have $|L(v_{m+1})|=|L(x)|=3$ and $v_{m+1}x\in M$ would
be at distance one from $y\in N$, which is a contradiction; therefore, $x\neq v_{m+2}$.  If $x=v_1$, then (A2) holds.
Finally, suppose that $x=v_{m+1}$.  In this case, $z\not\in V(F)$ has two neighbors in $\dom(\theta)$, and $X$ was
chosen according to (X2) or (X4).  However, then $|L(v_m)|\ge 4$, hence
$\deg(v_m)\ge 4$ and $v_m$ is adjacent to $y$ by Lemma~\ref{lemma-cl-basic}, and (A7) holds.
Hence, we can assume that $B'$ satisfies (D).

By the minimality of $B$, we conclude that $B'$ violates condition (O), and thus $B$ contains a near-obstruction $H$.
By Lemma~\ref{lemma-cl-nearobst}, $H$ is one of \ob{M1}, \ob{N2}, \ob{N3} or \ob{P3}.
\begin{itemize}
\item If $H$ is \ob{M1},
then let $xy$ be the edge of $H$ that belongs to $M'$, where $x$ is adjacent to $p_2$.
Note that $x,y\not\in V(F)$ by Lemma~\ref{lemma-cl-nearobst} and $xy\not\in M$.  If $x\not\in N$, then
$x$ has two neighbors $v_i$ and $v_j$ in $\dom(\theta)$, where $i<j$.  By Lemma~\ref{lemma-cl-2chord} applied to $p_2xv_i$,
we have $j=i+1$ and by the choice of $X$, $|L(v_j)|=4$; hence $\deg(v_j)\ge 4$, and by Lemma~\ref{lemma-cl-basic},
$v_j$ is incident with a crossing and thus $y\not\in N$.
Consequently, $y$ is also adjacent to $v_i$ and $v_j$.
However, note that $|L(v_i)|=3$, contradicting (C) for $G$.  Therefore, $x\in N$ is adjacent to $v_j$, and
$y$ is adjacent to both $v_i$ and $v_j$.  By Lemma~\ref{lemma-cl-2chord} applied to $p_0yv_j$, we have $i=1$, $j=2$
and $|L(v_1)|=3$, and by Lemma~\ref{lemma-cl-2chord} applied to $p_2xv_2$, we have that $s=3$ and $|L(v_3)|=3$. However, then $G$ is $L$-colorable.

\item Next, suppose that $H$ is \ob{N2} and let $x$ and $y$ be the vertices in the outer face of $H$ such that $|L_\theta(x)|=3$ and $|L_\theta(y)|=4$. By Lemma~\ref{lemma-cl-nochord}, $y\notin V(F)$.
If $x\in V(F)$, then by Lemma~\ref{lemma-cl-2chord} we have $s\le 2$, which is a contradiction, hence $x\not\in V(F)$.
Thus $x$ has two neighbors in $\dom(\theta)$ and $y$ has one, and by Lemma~\ref{lemma-cl-2chord} we conclude that $s=3$
and $|L(v_1)|=|L(v_3)|=3$. It follows that $X=\{v_1,v_2\}$, $x$ is adjacent to $v_1$ and $v_2$, and $y$ is adjacent to $v_2$.
There are two cases, either $v_2$ is incident with a crossed edge or $|N|=1$; in both of them, $G$ is $L$-colorable.

\item If $H$ is \ob{N3}, then let $xyz$ be the path in the outer face of $H$ such that $|L_\theta(x)|=|L_\theta(z)|=3$, $|L_\theta(y)|=4$ and $z$ is adjacent to $p_1$.
By Lemma~\ref{lemma-cl-nochord}, $z\not\in V(F)$, thus $z$ has two neighbors $w_1,w_2\in\dom(\theta)$, and by
Lemma~\ref{lemma-cl-2chord}, we can assume that the neighbors of $w_1$ are $w_2$, $z$ and an endvertex of $P$, and that
$|L(w_1)|=3$.  Since $y\not\in V(F)$ and $|L_\theta(y)|=4$, $y$ is adjacent to $w_2$.  Since $x$ cannot have more than one neighbor in
$\dom(\theta)$, we have $x\in V(F)$.  If $xw_2\not\in E(G)$, then $|L(x)|=|L_\theta(x)|=3$, and thus $x$ has no neighbor with list of size three
(since $H$ is an obstruction in $B'$, $x$ is either adjacent to a vertex in $N$, or incident with a crossing, and thus it is not incident
with an edge of $M$).  Lemma~\ref{lemma-cl-2chord} applied to the $2$-chord $xyw_2$ implies that the edge $xy$ is crossed by an edge incident with $w_2$. 
However, then $\deg(w_2)=4$ and (C) implies that $|L(w_2)|=5$, which is a contradiction.
We conclude that $xw_2\in E(G)$.  By the choice of $X$, $|L(x)|=3$.  Again, we distinguish two cases depending on whether
$w_2$ is incident with a crossed edge (in this case $|L(w_2)|=5$ by (C)) or $|N|=1$.  In both cases, $G$ is $L$-colorable.

\item If $H$ is \ob{P3}, then two of the vertices of $H$ have two neighbors in $\dom(\theta)$, hence $G$ contains a crossing at distance at most one
from $P$, which contradicts the distance condition for $B$.
\end{itemize}
These contradictions imply that at least one of (A1)--(A7) holds.
\end{proof}

Each case among (A1)--(A7) in Lemma~\ref{lemma-cl-a} contains a special subgraph.
Thus, a minimal counterexample $B=(G,P,N,M,L)$ with $P=p_0\ldots p_k$ contains a special subgraph $S$ whose
distance from $p_0$ is at most $2+r(S)$.  Consequently, $\ell(P)=k=2$.  Next, we consider the probe
$(X',\theta', b)$ at the $p_k$ side of $P$.  This probe satisfies one of the conditions symmetric to (A1)--(A7)
(satisfying one of (A1)--(A7) with $v_i$ replaced by $v_{s+1-i}$, and $X$ by $X'$), which we will refer to as (A1')--(A7').
In particular, there exists a special subgraph $S'$ whose distance from $p_2$ is at most $2+r(S')$.  It follows that
$d(S,S')\le 6+r(S)+r(S')$, and by (D), we conclude that $S=S'$.

To avoid repeating the definitions each time, let us fix the following notation ($\dagger$) for the following lemmas.
\begin{itemize}
\item $B=(G,P,N,M,L)$ is a minimal counterexample, with $P=p_0p_1p_2$.
\item $p_2p_1p_0v_1v_2\ldots v_s$ are the vertices contained in the boundary of the outer face $F$ of $G$ in the cyclic order.
\item If (A1) does not hold, then $(X,\theta,m)$ is a probe for $B$ at the $p_0$ side of $P$.
\item If (A1') does not hold, then $(X',\theta',b)$ is a probe for $B$ at the $p_2$ side of $P$.
\item $S$ is the special subgraph in $B$ at distance at most $4$ from $P$.
\end{itemize}

\begin{lemma}\label{lemma-cl-iscross}
With the notation {\rm ($\dagger$)}, $S$ is a crossing; hence, $B$ satisfies one of {\rm (A3)}--{\rm (A6)} and one of {\rm (A3')}--{\rm (A6')}.
\end{lemma}
\begin{proof}
Suppose for a contradiction that $B$ does not contain such a crossing.  Hence, it satisfies (A1), (A2), or (A7).

Suppose first that $B$ satisfies (A1), and thus $S$ is an edge of $M$ and $B$ also satisfies (A1').
It follows that $s\le 4$.  Since $s\ge 3$, we can by symmetry assume that $S=v_2v_3$.
If $v_2$, $v_3$ and $v_i$ have no common neighbor for $i\in \{1,4\}$ ($i=1$ if $s=3$), then let $\vf$ be an arbitrary $L$-coloring of $S$
(such that $\vf(v_3)\notin L(p_2)$ if $s=3$).  Let $B'=(\gf,P,N,\emptyset,\lf)$.
Observe that $B'$ cannot contain an obstruction since its special subgraph would be a special subgraph in $B$, too close to the special subgraph $S$.
Now it is easy to check using previously proved properties of $B$ that $B'$ satisfies all conditions of Theorem \ref{thm-maingen}.
(The same reasoning will be applied in the sequel without repeating it.)
Therefore, $B'$ with the list coloring $\lf$ is a counterexample to Theorem~\ref{thm-maingen}, contradicting the minimality of $B$.
Thus, by symmetry, we may assume that $v_1$, $v_2$ and $v_3$ have a common neighbor $w$.  In that case, $w$ is not adjacent to $v_4$ by Lemma~\ref{lemma-cl-2chord}.
Let $\vf$ be an $L$-coloring of $v_1$ and $v_3$ such that $\vf(v_1)\not\in L(p_0)$, $\vf(v_3)\not\in L(p_2)$ and $|\lf(v_2)|\ge 2$.
Observe that $(G-\{v_1,v_2,v_3\},P,N,\emptyset,\lf)$ is a counterexample contradicting the minimality of $B$,
since any $\lf$-coloring of $G-\{v_1,v_2,v_3)$ can be extended to $v_2$ by using a color in $\lf(v_2)$, and can henceforth be extended to $G$.
Therefore, $B$ does not satisfy (A1).

Let us now consider the case that $B$ satisfies (A2) or (A7), and thus $S\in N$.  Let $i$ and $j$ be the smallest and the largest integer,
respectively, such that $S$ is adjacent to $v_i$ and $v_j$.  By Lemma~\ref{lemma-cl-2chord} we have $j\in\{i+1,i+2\}$.
We consider the two possible values of $j$ separately:

\begin{itemize}
\item Suppose first that $j=i+1$.  If $|X|\ge 2$, then $|L(v_m)|\ge 4$ and $|L(v_{m-1})|=3$, hence (A7) and (A7') cannot both be true.
If both (A2) and (A2') hold, then since $s\ge 3$, we can assume that $v_2,v_3\in X$ have a common neighbor in $N$.
By the choice of $X$, we have $|L(v_4)|=3$, hence $s=4$ and $v_2,v_3\in X'$.  However, then $|L(v_1)|\ge 4$ by the choice of $X$
and $|L(v_1)|=3$ by the choice of $X'$, which is a contradiction.

Hence, we can assume that $B$ satisfies (A7) and (A2'); then we either have $s=m+1$, or we have $s=m+2$ and $X'=\{v_m, v_{m+1}\}$.
If there exists an $L$-coloring $\vf$ of $v_{m-1}$ and $v_{m+1}$ such that their colors are distinct from the colors of their neighbors
in $P$ and $|\lf(v_m)|\ge3$, then $B'=(G-\{v_{m-1},v_m,v_{m+1}\},P,N\setminus\{S\},\emptyset,\lf)$ contradicts the minimality of $B$;
observe that $B'$ satisfies (O), since no special subgraph of $G$ is at distance at most two from $S$.  A new special subgraph (an
edge joining two vertices with lists of size three) would appear in $B'$ only if $S$ were adjacent to
$v_{m+2}$, which is not the case since $j=i+1$.

We conclude that no such coloring exists.  Observe that this is only possible if both $v_{m-1}$ and $v_{m+1}$ have
a neighbor in $P$ (i.e., $m=2$ and $s=3$), $|L(v_1)|=|L(v_3)|=3$, and $L(v_2)$ is the disjoint union of
$L(v_1)\setminus L(p_0)$ and $L(v_3)\setminus L(p_2)$.
Let $w'$ be the common neighbor of $S$ and
$v_1$.  Suppose that there exists a color $c\in L(w')$ different from the colors of the neighbors of $w'$ in $P$
such that either $c\not\in L(v_2)$, or $v_1$ has degree three and $c\not\in L(v_1)\setminus L(p_0)$.
In this case, we let $\vf$ be the partial coloring such that $\vf(w')=c$ and let $G'=G-\{w',v_2\}$ if $c\not\in L(v_2)$
and $G'=G-\{w',v_1,v_2\}$ if $c\in L(v_2)$.  Observe that $G'$ is not $\lf$-colorable.  We conclude that
$(G',P,N,M\cup \{Sv_3\},\lf)$ is a counterexample contradicting the minimality of $B$ (the condition (O) holds by
Lemma~\ref{lemma-cl-nearobst}, the distance condition and Lemma~\ref{lemma-cl-nochord}).
Therefore, no such color $c$ exists.

Since $|L(w')|>|L(v_2)|$, it follows that
$w'$ has a neighbor in $P$.  By Lemma~\ref{lemma-cl-2chord}, $w'$ is not adjacent to $p_2$, hence it is adjacent to $p_0$ or
$p_1$.  However, then Lemmas~\ref{lemma-cl-nosepq} and \ref{lemma-cl-nochord} imply that $v_1$ has degree three, and
since $|L(v_1)\setminus L(p_0)|=2$ and $w'$ has at most two neighbors in $P$, the color $c$ exists.
This is a contradiction.

\item It remains to consider the case when $j=i+2$. In this case $S$ is adjacent to $v_i$ and $v_{i+2}$, and by Lemma~\ref{lemma-cl-2chord} we conclude that $v_{i+1}$ is a vertex of degree 3 with neighbors $v_i$, $v_{i+2}$, and $S$. Thus, $|L(v_{i+1})|=3$.
Suppose first that $B$ satisfies both (A7) and (A7'). If there exists a coloring $\vf$ of $S$ by a color
different from the colors of its neighbors in $P$ such that $\vf(S)\not\in L(v_i)\cap L(v_{i+1})\cap L(v_{i+2})$, then
$(G-\{S,v_i,v_{i+1}, v_{i+2}\},P, N\setminus\{S\},\emptyset,\lf)$ is a counterexample contradicting the minimality of $B$.
Otherwise, since $S$ is not adjacent to $p_0$ or $p_2$ by Lemma~\ref{lemma-cl-2chord}, we conclude that $S$ is adjacent to $p_1$ and
$L(S)\setminus L(p_1)=L(v_{i+1})\subseteq L(v_i)\cap L(v_{i+2})$.  However, in this case we let $\vf=\theta$ (where $\theta$ is
the coloring from the probe for $B$ at the $p_0$ side of $P$),
and note that $\vf(v_i)\not\in L(v_{i+1}) = L(S)\setminus L(p_1)$.  Hence, the construction of $\lf$ ensures that $|\lf(S)|=4$, and we conclude that $(G-X,P,N\setminus\{S\},\emptyset,\lf)$
is a counterexample contradicting the minimality of $B$.

Therefore, $B$ does not satisfy both (A7) and (A7'), and by symmetry, we can assume that $B$ satisfies (A2').
Let us first consider the case that $B$ satisfies (A2).  Note that
$v_{i+2}\not\in X$, as otherwise $|L(v_{i+3})|=3$ by the choice of $X$, and thus $v_{i+1}\not\in X'$, contradictory
to the assumption that $B$ satisfies (A2').  Symmetrically, $v_i\not\in X'$.  Since $|L(v_{i+1})|=3$, we cannot
have $\{v_i,v_{i+1}\}\subseteq X$, and since (A2) holds, we have $i=1$ and symmetrically, $s=3$.  Observe that we cannot color $S$ by a color
$\vf(S)\not\in L(v_{i+1})$, as otherwise $(G-\{S,v_{i+1}\},P,N\setminus \{S\},\emptyset,\lf)$ would contradict the minimality of $B$.
Therefore, $S$ has a neighbor in $P$, and by Lemma~\ref{lemma-cl-2chord}, this neighbor is $p_1$.  By Lemma~\ref{lemma-cl-nosepq},
the $4$-cycle $p_1p_0v_1S$ is not separating, and by Lemma~\ref{lemma-cl-nochord}, $v_1$ has degree three.  This contradicts Lemma~\ref{lemma-cl-basic}(a),
since $|L(v_1)|>3$.

Therefore, $B$ satisfies (A7).  Note that $v_{i+1}$ cannot be the element of $X'$ with the smallest index, and thus
$i+2=s$.  As before, we exclude the case that $S$ can be colored by a color not belonging to $L(v_i)\cap L(v_{i+1})$,
hence $S$ has a neighbor in $P$.  By Lemma~\ref{lemma-cl-2chord}, $S$ is adjacent
to $p_1$.  However, by Lemma~\ref{lemma-cl-nosepq}, the $4$-cycle $p_1Sv_{i+2}p_2$ is not separating, and by Lemma~\ref{lemma-cl-nochord},
$v_{i+2}$ is not adjacent to $p_1$.  Thus, $v_{i+2}$ has degree three and list of size at least four, which is a contradiction.
\end{itemize}

Therefore, $B$ does not satisfy (A2) or (A7), either.
\end{proof}

Note that if $B$ satisfies (A4) or (A4'), then $|V(G_q)\cap V(F)|=2$ by Lemma~\ref{lemma-cl-nochord}.  If $B$ satisfies (A6) or (A6'),
then by Lemmas~\ref{lemma-cl-nochord}, \ref{lemma-cl-nopcr} and \ref{lemma-cl-crossnn} we have $|V(G_q)\cap V(F)|=1$.  If $B$ satisfies
(A3) or (A3'), then by Lemmas~\ref{lemma-cl-2chord}, \ref{lemma-cl-nochord} and \ref{lemma-cl-nopcr}, we have $|V(G_q)\cap V(F)|\le 1$, and if $B$ satisfies (A5)
or (A5') then $1\le |V(G_q)\cap V(F)|\le 2$.

\begin{lemma}\label{lemma-noa3}
With the notation {\rm ($\dagger$)}, $B$ satisfies neither {\rm (A3)} nor {\rm (A3')}.
\end{lemma}
\begin{proof}
Suppose for a contradiction that $B$ satisfies (A3).  Let $w_1$ and $w_2$ be as in the description of (A3).
Note that $w_2$ is adjacent to $v_{m-1}$ and $v_m$ (even if $v_{m-1}\not\in\dom(\theta)$, in the case (X4b))
and that $|L(v_{m-1})|=|L(v_{m+1})|=3$.

Let $q$ denote the crossing contained in $S$. Let us first consider the case that $|V(G_q)\cap V(F)|=\emptyset$.
In this case $B$ satisfies (A3'), i.e., there exists $w_2'\in V(G_q)$ adjacent to $v_b$ and $v_{b+1}$,
and another vertex $w'_1$ of $G_q$ crossing-adjacent to $w'_2$ with a neighbor in $\dom(\theta')$.
Since $|L(v_b)|\neq 3$, we have $b\notin\{m-1,m+1\}$. Consequently, $|X\cap X'|\le 1$, and $w_2'\neq w_2$
by Lemma~\ref{lemma-cl-2chord}.

We now distinguish two cases regarding whether $w_2$ is adjacent or crossing-adjacent to $w'_2$ in $G_q$.
\begin{itemize}
\item Suppose that $w_2w'_2$ is a crossed edge.
Then $b\ne m$ by Lemma~\ref{lemma-cl-noseptr}
and the assumption that $G_q$ is disjoint with $F$; thus $b\ge m+2$.
Let $G_1$ and $G_2$ be the $v_mw_2w'_2v_b$-components of $G$,
such that $P\subset G_1$.  Since $w_1$ and $w_2$ are crossing-adjacent, we have $w_1\neq w'_2$, and symmetrically $w'_1\neq w_2$.
By Lemma~\ref{lemma-cl-2chord}, if $w_1=w'_1$, then $w_1$ belongs to $G_2$.  Hence, by symmetry,
we can assume that $w_1$ belongs to $G_2$, and thus $w_1$ is adjacent to $v_m$.

If $w_1$ is adjacent to $v_b$, then $b=m+2$ by Lemma~\ref{lemma-cl-2chord}.
Let $T=\{v_m,v_{m+1},v_{m+2},w_1\}$. By using Lemma~\ref{lemma-cl-nosepnearc} it is easy to see that
$|L(t)| = \deg(t)$ for each $t\in T\setminus \{w_1\}$ and that $\deg(w_1)\le 6$.
By the minimality of $B$, there exists an $L$-coloring $\vf$ of $G-T$.  Consider the subgraph $G'$ of $G$ induced by $T$
with the list assignment $\lf$.  We have $|\lf(v_{m+1})|=3$ and $|\lf(z)|\ge 2$ for $z\in T\setminus \{v_{m+1}\}$.
If $\lf(w_1)\neq \lf(v_m)$, then we color $w_1$ by a color in $\lf(w_1)\setminus \lf(v_m)$ and extend this coloring to
the rest of $G'$.  Similarly, $G'$ is $\lf$-colorable if $\lf(w_1)\neq \lf(v_{m+2})$.  If
$\lf(v_m)=\lf(w_1)=\lf(v_{m+2})$, then we color $v_{m+1}$ by a color in $\lf(v_{m+1})\setminus \lf(w_1)$ and
again we can extend this to an $\lf$-coloring of $G'$.  It follows that $G$ is $L$-colorable, which is a contradiction.

Therefore, $w_1$ is not adjacent to $v_b$, and in particular $w_1\neq w'_1$ and $w'_1\in V(G_1)$.
Let $\vf$ be an $L$-coloring of $G_1$, which exists by the minimality of $B$.
Since $w_1$ is not adjacent to $v_b$, note that $w_1$ has at most three neighbors in $G_1$ different from $w'_2$.
Hence, we can additionally choose a color $\vf(w_1)$ for $w_1$ different from the colors of its neighbors in $G_1$
so that $\vf(w_1)\neq \vf(w'_2)$. Let $G'_2=G_2-w_2+w_1w'_2$ and $P' = v_m w_1 w_2' v_b$, let $L'$
be the list assignment for $G'_2$ obtained from $L$ by setting $L'(z)=\{\vf(z)\}$ for $z\in V(P')$, and let
$B'_2=(G'_2, P', N\cap V(G'_2), M\cap E(G'_2), L')$.
Note that the added edge $w_1w_2'$ of $G'_2$ can be drawn without crossings following the crossed edges of $G$ that are no longer in
$G_2'$. Note that $\vf$ does not extend to an $L$-coloring of $G'_2$, and thus the target $B'_2$ is not valid.
Note that $B'_2$ may only violate (T) or (O).  In the former case, the vertex violating (T) must be $v_{m+1}$ and we would have $b=m+2$.
Consequently, $v_b$ would have degree three by Lemma~\ref{lemma-cl-nosepnearc}, which is a contradiction since $|L(v_b)|\ge 4$.  In the
latter case, since $|L(v_{m+1})|=|L(v_{b-1})|=3$ and $v_b$ has degree at least three in $G'_2$, we have that $G'_2$ is equal to \ob{P5} or \ob{P6}.
In both cases, any $L$-coloring of $G_1-\{v_m,v_b\}$ would extend to an $L$-coloring of $G$, a contradiction.

\begin{figure}
\begin{center}
\includegraphics[width=115mm]{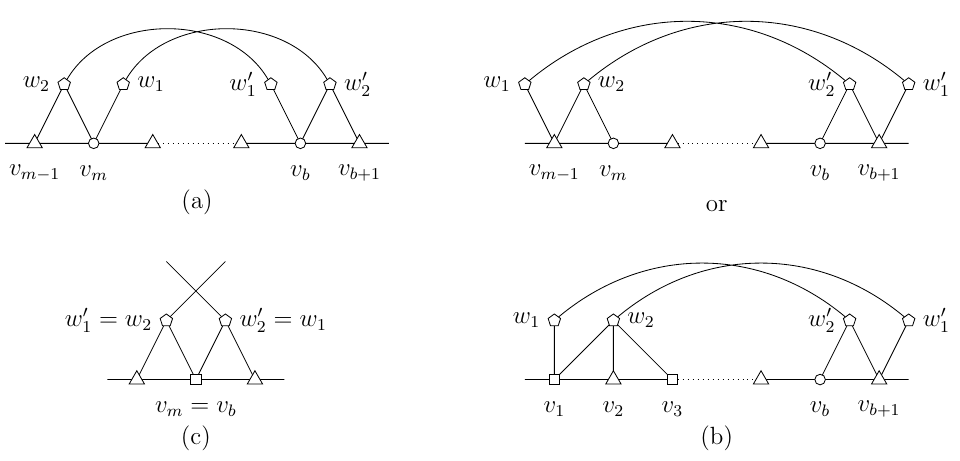}
\end{center}
\caption{Subcases when $w_2$ and $w_2'$ are crossing-adjacent.}
\label{fig-A4}
\end{figure}

\item Suppose now that $w_2$ is crossing-adjacent to $w'_2$.  Let $G_1$ and $G_2$ be the subgraphs of $G$ intersecting
in $\{v_b,w'_2,w_2,v_m\}$, where $P\subset G_1$ and $G_1\cup G_2$ is equal to $G-E(G_q)$.  We have two subcases:
either $b>m$ or $b=m$.
\begin{itemize}
\item If $b>m$, then Lemma~\ref{lemma-cl-2chord} implies that $w'_2$ has no neighbor in $X$, and thus $w_1\neq w'_2$.
Symmetrically, $w'_1\neq w_2$.  Since $w_1$ is crossing-adjacent to $w_2$, and $w'_1$ is crossing-adjacent to $w'_2$,
we conclude that the edges of $G_q$ are $w_1w'_2$ and $w'_1w_2$.

If $w_1,w'_1\notin V(G_1)$ (see Figure~\ref{fig-A4}(a)), then $w_1v_m,w'_1v_b\in E(G)$. Let $\vf$ be an $L$-coloring of
$G_1+\{w_1,w'_1,w_1w'_2,w'_1w_2,w_1w'_1\}$ which exists by the minimality of $G$, and note that $\vf$ does not extend to an $L$-coloring of $G'_2=G_2+w_1w'_1$.
Let $P'=v_mw_1w_1'v_b$ and let $L'$ be the list assignment for $G'_2$ obtained from $L$ by setting $L'(z)=\{\vf(z)\}$ for
$z\in V(P')$.  Observe that it is possible to choose $\vf$ so that $B'_2=(G'_2,P',N\cap V(G'_2), M\cap E(G'_2), L')$ is a counterexample
contradicting the minimality of $B$ (once the coloring of $G_1-\{v_m,v_b\}$ is fixed,
we still have two possible choices for the colors of $v_m$ and $v_b$ and three possible choices for the colors of $w_1$ and $w'_1$,
which suffices to ensure that $B'_2$ satisfies (T) and (O)).
This is a contradiction.  The case that $w_1,w'_1\in V(G_1)$ (see Figure~\ref{fig-A4}(b)) is excluded similarly.

\item If $b=m$, then let $w_2z$ and $w_2'z'$ be the edges of $G_q$ (note that we have $w_1=w'_2$ and $w'_1=w_2$).
Suppose that $z,z'\in V(G_2)$. Note that $V(G_2)\neq \{z,z',w_2,w'_2, v_m\}$, since otherwise $z$ would have degree
at most four and $|L(z)|=5$.  Therefore, the subgraph of $G$ induced by $V(G_1)\cup \{z,z'\}$ has an $L$-coloring $\psi$ by the minimality
of $G$.  Let $L'$ be the list assignment for $G'_2=G_2-\{z,z'\}$ obtained from $L$ by removing the colors of $z$ and $z'$ according
to $\psi$ from the lists of their neighbors and by setting $L'(w_2)=\{\psi(w_2)\}$, $L'(v_m)=\{\psi(v_m)\}$ and
$L'(w'_2)=\{\psi(w'_2)\}$.  Note that $(G'_2, w_2v_mw'_2, N\cap V(G_2),\emptyset,L')$ satisfies (O) and (C) by the distance
condition and (P) by the choice of $\psi$, and since $G$ is not $L$-colorable, we conclude that $B'_2$ violates (T).
Therefore, $G_2$ contains a vertex
adjacent to $w_2$, $w'_2$, $v_m$, $z$ and $z'$.  This contradicts Lemma~\ref{lemma-cl-noseptr}.

Therefore, we have $z,z'\in V(G_1)$, see Figure~\ref{fig-A4}(c). By Lemma~\ref{lemma-cl-nosepnearc}, $\deg(v_m)=4$.
Let $S_1=L(v_2)$ if $m=3$ and $S_1=L(v_1)\setminus L(p_0)$ if $m=2$.  Note that $S_1\subset L(v_m)$, as otherwise
we consider the partial coloring $\vf$ with $\vf(v_{m-1})\in S_1\setminus L(v_m)$ and conclude that
$(\gf, P, N, \emptyset, \lf)$ contradicts the minimality of $G$.

Suppose that there exists a color $c\in L(w_2)\setminus L(v_m)$,
or that $\deg(v_{m-1})=3$ and there exists a color $c\in L(w_2)\setminus S_1$, such that this color $c$ is distinct
from the colors of the neighbors of $w_2$ in $P$.  Let $\vf$ be the partial coloring with $\vf(w_2)=c$,
and let $G'=G-\{w_2,v_m\}$ if $\deg(v_{m-1})>3$ and $G'=G-\{w_2,v_m, v_{m-1}\}$ if $\deg(v_{m-1})=3$.
Note that $G'$ is not $\lf$-colorable, and that $\lf(v_{m-1})=L(v_{m-1})$ if $v_{m-1}$ belongs
to $V(G')$ since $\vf(w_2)\not\in S_1$.  By the minimality of $B$, the target $B'=(G',P, N\cup\{z\}, \emptyset,\lf)$ is not valid.
This is only possible if $B'$ violates (O), and by Lemma~\ref{lemma-cl-nearobst} and the distance condition, $G'$ contains \ob{N2}
or \ob{N3}.  However, then $z$ is adjacent to two vertices of $P$ and to $z'$ and $w'_2$, and at least one of $z'$ and $w'_2$
has a list of size three according to $L'$, which is a contradiction since $|L(z')|=|L(w'_2)|=5$.

We conclude that there exists no such color $c$.  Since $|L(v_m)|=4$ and $|L(w_2)|=5$, we conclude that $w_2$
has a neighbor in $P$.  By Lemma~\ref{lemma-cl-2chord}, $w_2$ is not adjacent to $p_2$, and if it were adjacent to $p_0$,
then we would have $m=2$, $\deg(v_1)=3$ and there would exist a color $c\in L(w_2)\setminus (S_1\cup L(p_0)\cup L(p_1))$.
Therefore, $w_2$ is adjacent to $p_1$.  By symmetry, $w'_2$ is adjacent to $p_1$ as well.  However, the
edges $w_2p_1$ and $w'_2p_1$ are not crossed by Lemma~\ref{lemma-cl-nopcr}, and thus the crossing is contained
inside the $4$-cycle $v_mw_2p_1w'_2$, contrary to Lemma~\ref{lemma-cl-nosepq}.
\end{itemize}
\end{itemize}

We conclude that $V(G_q)\cap V(F)\neq\emptyset$.  By Lemma~\ref{lemma-cl-nochord}, $w_2\not\in V(F)$.
Let $w$ be the vertex joined to $w_2$ by a crossed edge, and let $w_1w'$ be the other crossing edge.  Since $V(G_q)\cap X=\emptyset$, by
Lemmas~\ref{lemma-cl-nopcr} and \ref{lemma-cl-2chord} we have $w\not\in V(F)$.  Since $|L(v_m)|\ge 4$, $v_m$ has degree at least four; hence, we cannot have $w_1=v_{m+1}$, and by Lemmas~\ref{lemma-cl-nochord}
and \ref{lemma-cl-nopcr}, we have $w_1\not\in V(F)\setminus \{v_1\}$.  If $w_1\not\in V(F)$ and $x\in X$ is a neighbor of $w_1$,
then the $2$-chord $xw_1w'$ separates $P$ from either $w_2$ or $w$, and neither $w_2$ nor $w$ belongs to $F$, contrary to Lemma~\ref{lemma-cl-2chord}.
We conclude that $w_1=v_1$ and $V(G_q)\cap V(F)=\{v_1\}$, hence $v_1\not\in X$ and $X$ was chosen according to (X4a).

\begin{figure}
\begin{center}
\includegraphics{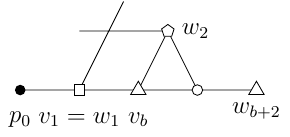}
\end{center}
\caption{Subcase combining (A3) and (A5').}
\label{fig-A4-A7}
\end{figure}

Since $|V(G_q)\cap V(F)|=1$, $B$ must satisfy (A3'), (A5') or (A6').
If $B$ satisfied (A3'), the conclusions of the preceding paragraph would apply symmetrically and we would have $v_1=v_s$, which is a contradiction.
Similarly, $X'$ cannot satisfy (A6').  The remaining possibility is that (A5') holds for $X'$. Then $v_1=v_{b-1}$ and
$v_b=v_2$. The situation is shown in Figure \ref{fig-A4-A7}.  Since $X$ was chosen according to (X4a),
we have $|L(v_b)|=|L(v_{b+2})|=3$; in particular, $s\ge4$ and $b\le s-2$.  This is only possible if
$X'$ has been chosen according to (X4), but then $|L(v_b)|>3$.  This is a contradiction, showing that $B$ does not satisfy (A3).
Symmetrically, $B$ does not satisfy (A3').
\end{proof}

Next, we exclude the case that the crossing near $P$ is incident with two vertices of $F$.
\begin{lemma}\label{lemma-ex1}
With the notation {\rm ($\dagger$)}, let $q$ denote the crossing contained in $S$.  Then $|V(G_q)\cap V(F)|=1$.
\end{lemma}
\begin{proof}
Since $B$ does not satisfy (A3), if $|V(G_q)\cap V(F)|\neq 1$ then $|V(G_q)\cap V(F)|=2$ and $B$ satisfies (A4) or (A5).
Symmetrically, $B$ satisfies (A4') or (A5').
By Lemmas~\ref{lemma-cl-nochord} and \ref{lemma-cl-crossnn}, $V(G_q)\cap V(F)=\{v_{m+1},v_{m+2}\}$ and $v_{m+1}$ is crossing-adjacent to $v_{m+2}$.
Let $v_{m+1}w$ and $v_{m+2}w'$ be the crossed edges.
By symmetry, we can assume that $|L(v_{m+1})|\ge |L(v_{m+2})|$.  By (C), either $|L(v_{m+1})|\ge |L(v_{m+2})|\ge4$ or $|L(v_{m+1})|=5$ and $|L(v_{m+2})|=3$.
Therefore, $X$ was chosen according to the rules (X1) or (X3) and $|L(v_m)|=3$.

If $L(v_{m+2})\neq L(v_{m+1})$, then let $c$ be a color in $L(v_{m+1})\setminus L(v_{m+2})$.  If $v_{m+1}$ is not adjacent to $v_{m+2}$, then
let $c$ be an arbitrary color in $L(v_{m+1})$. In both cases, let us define a partial $L$-coloring $\vf$ as follows.
Let $\vf(v_{m+1})=c$.  If $m=1$, then choose $\vf(v_1)\in L(v_1)\setminus L(p_0)$ distinct from $c$.
If $m=2$ and $v_1$, $v_2$, and $v_3$ have a common neighbor, then choose $\vf(v_1)\in L(v_1)\setminus L(p_0)$ so that $|\lf(v_2)|\ge 2$.
Otherwise, choose $\vf(v_m)\in L(v_m)$ distinct from $c$.  Note that $\gf-v_m$ is not $\lf$-colorable, and by the minimality of $B$,
we conclude that $B'=(\gf-v_m,P,N\cup\{w\},\emptyset,\lf)$ is not a valid target.
This is only possible if $B'$ violates (O).  By Lemma~\ref{lemma-cl-nearobst}, $B'$ contains \ob{N2} or \ob{N3}.
It follows that $w$ is adjacent to $p_1$ and to $p_0$ or $p_2$.  However, if $w$ is adjacent to $p_0$, then by Lemma~\ref{lemma-cl-2chord},
$v_{m+2}$ is incident with a chord of $F$, contradicting Lemma~\ref{lemma-cl-nochord}.  
If $w$ is adjacent to $p_2$, then $v_{m+2}=v_s$ by Lemma~\ref{lemma-cl-2chord}.
This contradicts the assumption that $B$ satisfies (A4') or (A5'). We conclude that $L(v_{m+1})=L(v_{m+2})$ (and in particular, $|L(v_{m+1})|=|L(v_{m+2})|=4$ by (A4) and (A5)),
and $v_{m+1}v_{m+2}\in E(G)$. By the choice of $X'$, we have $|L(v_{m+3})|=3$.

Suppose now that $w'v_m\in E(G)$. Note that $v_{m+1}$ has degree at least four, so it is adjacent to $w'$.
Let $S_1=L(v_m)$ if $m\neq 1$ and $S_1=L(v_m)\setminus L(p_0)$ if $m=1$.  Note that $S_1\subseteq L(v_{m+1})$, as otherwise we can choose an $L$-coloring $\vf$
of $v_m$ such that $\vf(v_m)\in S_1\setminus L(v_{m+1})$, and $(G-\{v_m,v_{m+1}\},P,N,\emptyset,\lf)$ is a counterexample contradicting the minimality of $B$.
Since $L(v_{m+1}) = L(v_{m+2})$, we conclude that $S_1\subseteq L(v_{m+2})$.
Let $G'$ be the graph obtained from $G-v_{m+1}$ by identifying $v_m$ with $v_{m+2}$ to a new vertex $z$, and let $L'$ be the list assignment obtained from $L$ by setting $L'(z)=L(v_m)$.
Let $B'=(G',P,N,\{zv_{m+3}\},L')$.  Observe that every coloring of $G'$ gives rise to an $L$-coloring of $G$, and thus $G'$ is not
$L'$-colorable.  Note that $B'$ satisfies (O), since $B'$ contains neither \ob{M1} nor \ob{M2}, and the exclusion of other obstructions is obvious.
It follows that $B'$ is a counterexample contradicting the minimality of $B$.
Therefore, $w'v_m\not\in E(G)$, and by symmetry, $wv_{m+3}\not\in E(G)$.

Let $S_2=L(v_{m+3})$ if $m+3\neq s$ and $S_2=L(v_{m+3})\setminus L(p_2)$ if $m+3=s$.  Suppose now that there exists an $L$-coloring
$\vf$ of $v_{m+1}$ and $v_{m+2}$ by distinct colors such that $\vf(v_{m+1})\not\in S_1$ and $\vf(v_{m+2})\not\in S_2$.
Let $M'=\{ww'\}$ if $ww'\in E(G)$, and $M'=\emptyset$ otherwise.
By the minimality of $B$, we conclude that $B'=(\gf,P, N, M', \lf)$ is not a valid target.
This is only possible if $B'$ violates (O).
By Lemma~\ref{lemma-cl-nearobst}, $B'$ contains \ob{M1} (the other cases are easily excluded: \ob{N2} and \ob{N3} since no internal vertex
gets a reduced list and \ob{P3} since $\ell(P)=2$).  But then $w'$ is adjacent to $p_0$,
and the $2$-chord $p_0w'v_{m+2}$ contradicts Lemma~\ref{lemma-cl-2chord}.  Therefore, no such coloring $\vf$ exists.
It follows that $|S_1|=|S_2|=3$ and $S_1\subseteq L(v_{m+1})$.
Since $L(v_{m+1})=L(v_{m+2})$, we also have that $S_1=S_2$.
Since $|S_1|=|S_2|=3$, claim Lemma~\ref{lemma-cl-basic}(f) implies that $m=2$ and $s=6$.
Similarly, we conclude that $L(v_1)=L(p_0)\cup L(v_2)$ and $L(v_6)=L(p_2)\cup L(v_5)$, as otherwise we can color and remove $v_1$ or $v_6$.

Let us now consider the case that $v_2$, $v_3$ and $w'$ have no common neighbor.  If $v_1$, $v_2$ and $v_3$ have no common neighbor, then let
$\vf$ be an $L$-coloring of $v_2$, $v_3$ and $v_4$ such that $\vf(v_4)\not\in L(v_5)$.  Otherwise, let $\vf$ be an $L$-coloring
of $v_1$, $v_3$ and $v_4$ such that $\vf(v_4)\not\in L(v_5)$ and $\vf(v_1)=\vf(v_3)$.  In the former case, let $G'=\gf$,
in the latter case let $G'=\gf-v_2$.  Let $M'=\{ww'\}$ if $ww'\in E(G)$, and $M'=\emptyset$ otherwise.
Note that $B'=(G',P,N,M',\lf)$ satisfies (O) by Lemma~\ref{lemma-cl-nearobst}, since $w'$ cannot be adjacent to $p_0$.
We conclude that $B'$ is a counterexample contradicting the minimality of $B$.
Therefore, $v_2$, $v_3$ and $w'$ have a common neighbor $x'$,
and by symmetry, $v_4$, $v_5$ and $w$ have a common neighbor $x$ (see Figure~\ref{fig-A6}).

\begin{figure}
\begin{center}
\includegraphics[width=78mm]{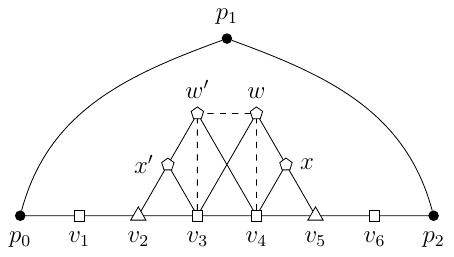}
\end{center}
\caption{A subcase in the proof when $|V(G_q)\cap V(F)|=2$. The dotted edges may or may not be present.}
\label{fig-A6}
\end{figure}

By Lemma~\ref{lemma-cl-2chord}, we have $x\neq x'$ and
$x$ is adjacent neither to $p_0$ nor to $p_2$.  Furthermore, if $xp_1\in E(G)$, then consider the cycle $K=p_1p_2v_6v_5x$.  Since $v_6$ has degree at
least four, we conclude by Lemma~\ref{lemma-cl-noseppen} that $K$ has two chords incident with $v_6$.  However, that contradicts Lemma~\ref{lemma-cl-nochord}.
Therefore, $x$ (and symmetrically $x'$) has no neighbor in $P$.  By Lemma~\ref{lemma-cl-2chord}, neither $w$ nor $w'$ is adjacent to $p_0$ or $p_2$.
Lemmas~\ref{lemma-cl-noseptr} and \ref{lemma-cl-nosepq} imply that $x'w, xw', xx'\not\in E(G)$.
Since both $w$ and $w'$ have degree at least $5$, we conclude that each of them is adjacent either to $p_1$ or to a vertex not shown in Figure~\ref{fig-A6}.
Suppose that $w'p_1\not\in E(G)$.  Then let $\vf$ be an $L$-coloring of $x$ and $w'$ such that $\vf(x),\vf(w')\not\in L(v_4)$
(note that these colors do not belong to the lists of $v_2$, $v_3$ and $v_5$, as well as to $L(v_1)\setminus L(p_0)$ and $L(v_6)\setminus L(p_2)$).
Let $G'=G-\{x,w',v_3,v_4\}$ if $\deg(w)>5$ and $G'=G-\{x,w',v_3,v_4,w\}$ if $\deg(w)=5$.
Note that $G'$ is not $\lf$-colorable since any $\lf$-coloring of $G'$ extends to $G$. Furthermore, the only possible vertices
with list of size three in $G'$ are $v_2$, $v_5$, $w$ and a common neighbor $u$ of $x$ and $w'$ distinct from $w$ and $v_4$, if such a vertex exists.
By Lemma~\ref{lemma-cl-nosepq}, if $u$ exists, then $\deg(w)=5$ and $w\not\in V(G')$.  Furthermore, by Lemma~\ref{lemma-cl-nosepq},
$u$ and $w$ are not adjacent to $v_2$ and $v_5$.  Therefore, $(G',P,N,\emptyset,\lf)$ is a counterexample contradicting the
minimality of $B$.

We conclude that $w'p_1\in E(G)$.  Let $G_1$ and $G_2$ be the $p_1w'v_4$-components of $G$, where $G_1$ contains $p_0$.  Consider an $L$-coloring $\vf$ of
$G_2$. Note that $v_3$ has only two neighbors in $G_2-w'$, and thus $\vf$ can be extended to $v_3$ in such a way that $\vf(v_3)\neq\vf(w')$.
Let $G'=G_1-v_4+w'v_3$, $P'=p_0p_1w'v_3$ and let $L'$ be the list assignment obtained from $L$ by setting $L'(z)=\{\vf(z)\}$ for $z\in V(P')$.
By the minimality of $B$, we conclude that $B'=(G',P',N,\emptyset,L')$ is not a valid target.
This is only possible if $B'$ violates (O). Observe that only $v_1$ and $v_2$ and vertices in $N$ have list of size at most four in $B'$ and that
$x'$ is a common neighbor of $v_3$ and $w'$.  Therefore, $x'$ is a vertex in the corresponding obstruction $K$, and $v_2$ is a vertex
in $K$ with list of size 3. It follows that $K$ is equal to \ob{P4} or \ob{P6}.  However, then $v_1p_1\in E(G)$ or $v_2p_0\in E(G)$,
contradicting Lemma~\ref{lemma-cl-nochord}.
\end{proof}

We are now ready to finish the proof of the main theorem.

\begin{proof}[Proof of Theorem~\ref{thm-maingen}]
We proceed by contradiction; if Theorem~\ref{thm-maingen} is false, then there exists a counterexample, and we
can choose $B$ as a minimal counterexample.
With the notation ($\dagger$), note that $S$ is a crossing by Lemma~\ref{lemma-cl-iscross}.  Let $q$ be the crossing contained in $S$. 
By Lemma~\ref{lemma-ex1}, we have $|V(G_q)\cap V(F)|=1$, and thus $B$ satisfies (A5) or (A6), and (A5') or (A6').  Since $s\ge 3$, we can
by symmetry assume that $B$ satisfies (A5').

Suppose first that $B$ satisfies (A6), and thus $b=2$.  Let $z$ be the neighbor of $v_1$ along the crossed edge.
Since $v_1\not\in X$, the inspection of possible cases for $X$ and $X'$ shows that we have $|L(v_2)|=3$, $X'=\{v_2\}$, and $s=3$.  If $v_1$, $v_2$ and
$v_3$ have no common neighbor, then consider any $L$-coloring $\vf$ of $v_1$ and $v_2$ such that $\vf(v_1)\not\in L(p_0)$, and observe that
$(\gf, P, N\cup\{z\}, \emptyset,\lf)$ is a counterexample contradicting the minimality of $B$ (since $v_1,v_2,v_3$ do not have a common neighbor, we do not get
adjacent vertices with lists of size 3, and (O) is satisfied since $z$ is not adjacent to $p_0$ and $p_2$ by Lemmas~\ref{lemma-cl-2chord} and \ref{lemma-cl-nopcr}).

Hence, we can assume
that $v_1$, $v_2$ and $v_3$ have a common neighbor $w$, and thus $\deg(v_2)=3$.
Similarly, we conclude that $L(v_1)=L(p_0)\cup L(v_2)$ (if not, we color $v_1$ with a color
in $L(v_1)\setminus (L(p_0)\cup L(v_2))$ and then consider $G'=G-\{v_1,v_2\}$) and that $L(v_3)=L(p_2)\cup L(v_2)$ (if not, we can color $v_3$
by a color in $L(v_3)\setminus (L(p_2)\cup L(v_2))$ and then consider $G'=G-\{v_2,v_3\}$).
By Lemmas~\ref{lemma-cl-nosepq}, \ref{lemma-cl-nopcr} and \ref{lemma-cl-2chord}, $w$ has no neighbor in $P$.  Let $u$ be the vertex adjacent to $w$ by the
crossed edge, let $\vf$ be an $L$-coloring of $w$ such that $\vf(w)\not\in L(v_2)$ and let $G'=G-\{v_2,w\}$.
Note that $B'=(G',P,N\cup \{u\},\emptyset,\lf)$ satisfies (O), since it contains no vertex with list of size three.
Thus, $B'$ is a counterexample to Theorem~\ref{thm-maingen} contradicting the minimality of $G$,
which implies that $B$ does not satisfy (A6).

Therefore, $B$ satisfies both (A5) and (A5'), and we have $b=m+2$.  Moreover, Lemma~\ref{lemma-cl-nosepq} implies that the neighbors $w$ and $w'$ of $v_m$ and $v_b$ in $V(G_q)\setminus \{v_{m+1}\}$
are distinct.  Let $y$ be the vertex joined to $v_{m+1}$ by a crossed edge.
If $|L(v_{m+1})|\ne3$, then both $X$ and $X'$ are chosen by cases (X1) or (X3) and $|L(v_m)|=|L(v_{m+2})|=3$.
The condition (A5) implies $|L(v_{m+1})|=4$. However, in that case we have $|L(v_{m+2})|\ne3$ both in (X1) and (X3), which is a contradiction.
Therefore, $|L(v_{m+1})|=3$.  Consequently, $X$ and $X'$ were chosen by (X2) or (X4) and we have
$|L(v_m)|,|L(v_{m+2})|\ge 4$ and $|L(v_{m-1})|=|L(v_{m+3})|=3$.
Since $\deg(v_m)\ge 4$, Lemmas~\ref{lemma-cl-nosepnearc} and \ref{lemma-cl-2chord} imply that $w$ has no neighbor in $F$ other than $p_1$, $v_m$ and $v_{m+1}$,
and by symmetry, the only possible neighbors of $w'$ in $F$ are $p_1$, $v_{m+1}$ and $v_{m+2}$.

Let $S_1=L(v_{m-1})$ if $m=3$ and $S_1=L(v_{m-1})\setminus L(p_0)$ if $m=2$.
Let $S_2=L(v_{b+1})$ if $b=s-2$ and $S_2=L(v_{b+1})\setminus L(p_2)$ if $b=s-1$.
By symmetry, we can assume that if $m=2$, then $b=s-1$.
Let $S$ be the set of colors $c\in L(v_{m+1})$ such that either
\begin{itemize}
\item[{\rm (a)}] $L(v_{m+2})=S_2\cup \{c\}$, or
\item[{\rm (b)}] $|L(v_m)|=4$, $c\not\in S_1$ and $S_1\cup \{c\}\subseteq L(v_m)$.
\end{itemize}
If $m=2$, then we have $b=s-1$, $|S_1|=|S_2|=2$, there are at most two colors with the property (b)
and no colors with the property (a).
If $m=3$, then $|S_1|=3$ and $|S_2|\le 3$, there is at most one color with the property (b)
and at most one color with the property (a).  It follows that $|S|\le 2$.
Let $\vf$ be an $L$-coloring of $v_{m-1}$, $v_{m+1}$ and $v_{m+2}$ chosen so that
$\vf(v_{m+2})\not\in S_2$, $\vf(v_{m+1})\not\in S$, $\vf(v_{m-1})\in S_1$ and
$|L(v_m)\setminus \{\vf(v_{m-1}),\vf(v_{m+1})\}|\ge 3$.
Note that the choices for $\vf(v_{m+2})$ and $\vf(v_{m-1})$ are possible,
since $\vf(v_{m+1})$ does not satisfy (a) and (b), respectively.

Consider $G'=G-\{v_{m-1}, v_{m+1}, v_{m+2}\}$ with the list assignment $\lf$.
By Lemma~\ref{lemma-cl-2chord}, $v_{m-1}$ has no common neighbor with $v_{m+1}$ other than $v_m$,
and $v_{m-1}$ has no common neighbor with $v_{m+2}$, and the only common neighbor of $v_{m+1}$ and $v_{m+2}$
is $w'$.  Therefore, the only vertices with list of size three are
$v_1$ if $m=3$, $v_m$, $v_{m+3}$ and $w'$.  Since $w'$ is not adjacent to
$v_{m+3}$, we conclude that $B'=(G',P,N\cup \{y\},\emptyset,\lf)$ satisfies (M).
Furthermore, $y$ is adjacent neither to $p_0$ nor to $p_2$ by Lemma~\ref{lemma-cl-2chord}, hence $B'$
satisfies (O) by Lemma~\ref{lemma-cl-nearobst}.  Therefore, $B'$ is a counterexample contradicting the minimality of~$B$.

This shows that there exists no minimal counterexample, and thus Theorem~\ref{thm-maingen} holds.
\end{proof}

\section{Acknowledgments}
We would like to thank Riste \v{S}krekovski for hosting us during our stay in Ljubljana where this paper originated
and for stimulating discussions of the problem.  We would also like to thank the referee of the paper for persistence
and useful comments.

\bibliographystyle{elsarticle-num}
\bibliography{5choosfar}

\end{document}